\documentclass[11pt]{amsart}%
\usepackage{graphicx}
\usepackage{amscd}
\usepackage{amsmath}
\usepackage{amsfonts}
\usepackage{amssymb}%
\providecommand{\U}[1]{\protect\rule{.1in}{.1in}}
\providecommand{\U}[1]{\protect\rule{.1in}{.1in}}
\newtheorem{theorem}{Theorem}
\theoremstyle{plain}

\newtheorem{condition}{Condition}[section]

\newtheorem{definition}{Definition}[section]

\newtheorem{lemma}{Lemma}[section]

\newtheorem{proposition}{Proposition}[section]
\newtheorem{remark}{Remark}[section]

\numberwithin{equation}{section}
\numberwithin{theorem}{section}

\begin{document}
\title[Self-Adjoint Operators in Extended Hilbert Spaces]{Self-Adjoint Operators in Extended Hilbert Spaces $H\oplus W$: An Application
of the General GKN-EM Theorem}
\author{Lance L. Littlejohn and Richard Wellman}
\address{Department of Mathematics, Baylor University, Waco, TX 76798-7328 \\
Department of Mathematics and Computer Science, Westminster College, Salt Lake
City, UT 84105 }
\email{rwellman@westminstercollege.edu}
\date{April 23, 2017 [wellmanlittlejohnEMfinalversion.tex]}

\begin{abstract}
We construct self-adjoint operators in the direct sum of a complex Hilbert
space $H$ and a finite dimensional complex inner product space $W$. The
operator theory developed in this paper for the Hilbert space $H\oplus W$ is
originally motivated by some fourth-order differential operators, studied by
Everitt and others, having orthogonal polynomial eigenfunctions. Generated by
a closed symmetric operator $T_{0}$ in $H$ with equal and finite deficiency
indices and its adjoint $T_{1}$, we define \textit{families} of minimal
operators $\{\widehat{T}_{0}\}$ and maximal operators $\{\widehat{T}_{1}\}$ in
the extended space $H\oplus W$ and establish, using a recent theory of complex
symplectic geometry, developed by Everitt and Markus, a characterization of
self-adjoint extensions of $\{\widehat{T}_{0}\}$ when the dimension of the
extension space $W$ is not greater than the deficiency index of $T_{0}$. A
generalization of the classical Glazman-Krein-Naimark (GKN) Theorem - called
the GKN-EM Theorem to acknowledge the work of Everitt and Markus - is key to
finding these self-adjoint extensions in $H\oplus W.$ We consider several
examples to illustrate our results.

\end{abstract}
\maketitle
\tableofcontents

\section{Introduction}

In \cite[p. 105]{ELW-Patras}, the authors list ten open problems related to
orthogonal polynomial eigenfunctions of differential equations. The one
pertinent to this present paper is the following (paraphrased to simplify the
original notation):

\begin{quotation}
The GKN theory provides a recipe, in theory, for determining all self-adjoint
extensions in the Hilbert space $L^{2}(I;w)$ of formally symmetric
differential expressions of the form%
\begin{equation}
\ell_{2r}[y](u)=\dfrac{1}{w(u)}\sum_{j=0}^{r}(-1)^{j}(q_{j}(u)y^{(j)}%
(u))^{(j)}\quad(u\in I) \label{Symmetric DE}%
\end{equation}
on some open interval $I=(a,b);$ we assume here that $w>0$ and each
coefficient $q_{j}$ is sufficiently differentiable on $I.$ This theory works
well in developing the spectral theory for the second-order classical
differential equations of Jacobi, Laguerre, and Hermite.\footnote{The GKN
theory is applicable as well to developing the self-adjoint operator theory
associated with exceptional orthogonal polynomials.} However, for nonclassical
symmetric differential equations (\ref{Symmetric DE}) with orthogonal
polynomial solutions$,$ the appropriate right-definite setting is a
Hilbert-Sobolev space $S$ with orthogonalizing Sobolev inner product%
\begin{equation}
\left\langle f,g\right\rangle =\int_{a}^{b}f(u)\overline{g}(u)w(u)du+\sum
_{j=0}^{p}\left(  \alpha_{j}f^{(j)}(a)\overline{g}^{(j)}(a)+\beta_{j}%
f^{(j)}(b)\overline{g}^{(j)}(b)\right)  . \label{General Orthogonalizing IP}%
\end{equation}
The Sobolev space $S$ has the form $L^{2}(I;w)\oplus\mathbb{C}^{k}$ for some
$k\leq2p.$ Develop a general GKN-type theory for this setting; in particular,
provide a `recipe' for determining the self-adjoint operator having the
orthogonal polynomials as eigenfunctions.
\end{quotation}

In this paper, we answer this question. In fact, we will see that we can
provide a recipe for \textit{all} self-adjoint operators, generated by
$\ell_{2r}[\cdot],$ in this Sobolev setting. Our result is a generalization of
the Glazman-Krein-Naimark (GKN) theory of self-adjoint extensions of
Lagrangian symmetric ordinary differential expressions in a weighted Hilbert
space $L^{2}(I;w),$ where $I$ is an interval of the real line $\mathbb{R}.$

This work is originally motivated by fourth-order differential equations
having non-classical orthogonal polynomials as eigenfunctions. In each of
these fourth-order examples, the orthogonalizing inner product has the form%
\[
\left\langle f,g\right\rangle =Af(a)\overline{g}(a)+\int_{a}^{b}%
f(u)\overline{g}(u)w(u)du+Bf(b)\overline{g}(b),
\]
where $A,B\geq0.$ Indeed, H. L. Krall \cite{HLKrall-1938, HLKrall-1940}
classified, up to a complex linear change of variable, these orthogonal
polynomials which were subsequently named the Legendre type, Laguerre type and
Jacobi type polynomials and studied extensively by A. M. Krall in
\cite{AMKrall-1981}. Following the work of the two Kralls, other contributions
connecting orthogonal polynomial eigenfunctions to higher-order differential
equations have emerged; all known examples have polynomial eigenfunctions
orthogonal with respect to an inner product of the form
(\ref{General Orthogonalizing IP}). These various contributions are far too
numerous to list in this manuscript but we refer to the Erice and Patras
reports \cite{EL-Erice, ELW-Patras} for further details and the references
therein contained.

In \cite{EKL-QM, EL-DIE}, the authors construct the Legendre type self-adjoint
operator, generated by the fourth-order Legendre type differential expression%
\begin{equation}
\ell_{LT}[y](u):=(u^{2}-1)^{2}y^{(4)}(u)+8u(u^{2}-1)y^{(3)}(u)+(4A+12)(u^{2}%
-1)y^{\prime\prime}(u)+8Auy^{\prime}(u), \label{legendre type equation}%
\end{equation}
in the Hilbert space $L_{\mu}^{2}[-1,1].$ Here%
\[
L_{\mu}^{2}[-1,1]=\{f:[-1,1]\rightarrow\mathbb{C}\mid f\text{ is Lebesgue
measurable and }\left\langle f,f\right\rangle _{\mu}=\int_{[-1,1]}\left\vert
f\right\vert ^{2}d\mu<\infty\},
\]
where%
\[
\left\langle f,g\right\rangle _{\mu}=Af(-1)\overline{g}(-1)+\int_{-1}%
^{1}f(u)\overline{g}(u)du+Af(1)\overline{g}(1).
\]
We note that $L_{\mu}^{2}[-1,1]$ is isometrically isomorphic to $L^{2}%
(-1,1)\oplus\mathbb{C}^{2}.$ The classical Glazman-Krein-Naimark (GKN) theory
of self-adjoint extensions of Lagrangian symmetric differential expressions is
not (immediately) applicable in this situation. To develop the appropriate
operator theory in $L_{\mu}^{2}[-1,1],$ Everitt and Littlejohn studied
properties of functions in the maximal domain $\Delta$ of $\ell_{LT}[\cdot]$
in the base space $L^{2}(-1,1).$ They prove the surprising smoothness
condition
\[
f\in\Delta\Rightarrow f^{\prime\prime}\in L^{2}(-1,1)
\]
from which it follows that $f,f^{\prime}\in AC[-1,1].$ Using standard
operator-theoretic methods, they then prove that the operator $T:\mathcal{D}%
(T)\subset L_{\mu}^{2}[-1,1]\rightarrow L_{\mu}^{2}[-1,1]$ defined by%
\begin{align*}
(Tf)(u)  &  =\left\{
\begin{array}
[c]{ll}%
-8Af^{\prime}(-1) & u=-1\\
\ell_{LT}[f](u) & -1<u<1\\
8Af^{\prime}(1) & u=1
\end{array}
\right. \\
\mathcal{D}(T)  &  =\Delta
\end{align*}
is self-adjoint. It is remarkable that $\Delta$ is the domain of the
self-adjoint operator $T$ in $L_{\mu}^{2}[-1,1].$ Indeed, the expression
$\ell_{LT}[\cdot]$ is in the limit-3 case at both singular endpoints $x=\pm1$
in $L^{2}(-1,1)$ so every self-adjoint operator in $L^{2}(-1,1),$ generated by
$\ell_{LT}[\cdot],$ is necessarily determined by two appropriate boundary
restrictions on the space $\Delta.$

We will re-examine this Legendre type example in Section \ref{Section Four} as
an application of the results developed in this paper. To this end, let
$(W,\left\langle \cdot,\cdot\right\rangle _{W})$ be a
\textit{finite-dimensional} complex inner product space and assume $\left(
H,\left\langle \cdot,\cdot\right\rangle _{H}\right)  $ is a complex Hilbert
space. Then $H\oplus W$, the direct sum of $H$ and $W,$ is the Hilbert space
defined by%

\begin{equation}
H\oplus W=\{(x,a)\mid x\in H,a\in W\} \label{H+W space}%
\end{equation}
with inner product%
\begin{equation}
\left\langle (x,a),(y,b)\right\rangle _{H\oplus W}:=\left\langle
x,y\right\rangle _{H}+\left\langle a,b\right\rangle _{W} \label{H+W IP}%
\end{equation}
and associated norm%
\[
\left\Vert (x,a)\right\Vert _{H\oplus W}^{2}=\left\Vert x\right\Vert _{H}%
^{2}+\left\Vert a\right\Vert _{W}^{2}.
\]
Throughout this paper, we refer to $H\oplus W$ as an \textit{extended Hilbert
space} and call $H$ the \textit{base} \textit{space }and $W$ the
\textit{extension space}.

Our starting point in this paper - assumptions we keep throughout this article
- is a closed, symmetric operator $T_{0}$ in $H$ having equal and finite
deficiency indices, denoted by their common value \textrm{def}$(T_{0}),$ and
adjoint operator $T_{1}$ satisfying the inclusions
\[
T_{1}^{\ast}=T_{0}\subseteq T_{0}^{\ast}=T_{1}.
\]
We call $T_{0}$ the \textit{minimal} operator and $T_{1}$ the \textit{maximal}
operator in $H.$ Then, under the essential assumption that
\[
\mathrm{def}(T_{0})\leq\dim W,
\]
we construct one-parameter families $\{\widehat{T}_{0}\}$ of \textit{minimal
}operators and associated \textit{maximal} operators $\{\widehat{T}_{1}\}$ in
$H\oplus W,$ generated by $T_{0}$ and $T_{1}$ in $H,$ satisfying the
properties%
\[
\mathrm{def}\text{ }(\widehat{T}_{0})=\mathrm{def}(T_{0})\quad(\widehat{T}%
_{0}\in\{\widehat{T}_{0}\})
\]
and%
\[
\left(  \widehat{T}_{1}\right)  ^{\ast}=\widehat{T}_{0}\subseteq\left(
\widehat{T}_{0}\right)  ^{\ast}=\widehat{T}_{1}\quad(\widehat{T}_{0}%
\in\{\widehat{T}_{0}\},\text{ }\widehat{T}_{1}\in\{\widehat{T}_{1}\}).
\]
Both families $\{\widehat{T}_{0}\}$ and $\{\widehat{T}_{1}\}$ are parametrized
by an arbitrary, fixed self-adjoint operator $B:W\rightarrow W.$

With the constructions of $\{\widehat{T}_{0}\}$ and $\{\widehat{T}_{1}\}$ in
place, we then appeal to a general theory of complex symplectic algebra, with
important applications and implications to boundary value problems in ordinary
and partial differential equations, which was developed by Everitt and Markus
in a series of remarkable papers \cite{EM1, EM2, EM3, EM4}. An important
consequence of their theory is a generalized GKN theory - which we call GKN-EM
theory after the contributions of Everitt and Markus - that we apply to
characterize all self-adjoint extensions (respectively, restrictions) of
$\widehat{T}_{0}\in\{\widehat{T}_{0}\}$ (respectively, of $\widehat{T}_{1}%
\in\{\widehat{T}_{1}\}$).

The contents of this paper are as follows. In Section \ref{Section Two}, we
briefly discuss the Stone-von Neumann theory of self-adjoint extensions of
symmetric operators in a Hilbert space as well as the now classic GKN theory,
including a statement of the GKN Theorem (Theorem \ref{The GKN Theorem}).
Section \ref{Section Two.Five} deals with key complex symplectic geometric
results developed by Everitt and Markus and culminates in the GKN-EM Theorem
(Theorem \ref{The GKN-EM Theorem}). The families $\{\widehat{T}_{0}\}$ and
$\{\widehat{T}_{1}\}$ of minimal and maximal operators in $H\oplus W,$
generated by $T_{0}$ and $T_{1}$ in the base space $H,$ are developed in
Section \ref{Section Three}. Another key notion, the symplectic form
$[\cdot,\cdot]_{H\oplus W}$ in the extended space $H\oplus W$, essential to
our application of the GKN-EM theory, is defined in Section
\ref{Section Three}. Also, in this section, we apply Theorem
\ref{The GKN-EM Theorem} to characterize all self-adjoint extensions
$\{\widehat{T}\}$ of $\widehat{T}_{0}\in\{\widehat{T}_{0}\}$; see the summary
theorem given in Theorem \ref{Self-Adjoint Operators in H+W}. Lastly, Section
\ref{Examples} deals with several examples to illustrate our results. These
examples include another look at the Legendre type example where further light
is shed on this particular example. Indeed, we show that, remarkably,
\textit{continuity} is a GKN-EM boundary condition. \bigskip

\underline{Notation}: $\mathbb{R},\mathbb{C}$ and $\mathbb{N}$ will denote,
respectively, the sets of real numbers, the complex numbers and the positive
integers. All inner products in this paper will be denoted by $\left\langle
\cdot,\cdot\right\rangle ,$ properly subscripted indicating the particular
underlying vector space. Ordered pairs in $H\oplus W$ will be written as
$\left(  \cdot,\cdot\right)  ;$ if $\dim W>1,$ then an ordered pair in
$H\oplus W$ will have the form $(\cdot,(\cdot,\cdot)).$ Our \textit{base}
space will be a complex Hilbert space $(H,\left\langle \cdot,\cdot
\right\rangle _{H})$, our \textit{extension} space will be finite-dimensional
complex Hilbert space $(W,\left\langle \cdot,\cdot\right\rangle _{W})$ and the
\textit{extended} space will be the direct sum space $(H\oplus W,\left\langle
\cdot,\cdot\right\rangle _{H\oplus W}).$ Linear operators in the base space
$H$ will be denoted by $T_{0},$ $T_{1},$ $T,$ etc. while operators in the
extended space $H\oplus W$ will be hatted: $\widehat{T}_{0},$ $\widehat{T}%
_{1},$ $\widehat{T},$ etc. The notation%
\[
x\text{ has property }P\quad(x\in A)
\]
means that property $P$ holds for all $x$ in the set $A.$ Lastly, the
cardinality of a set $A$ is denoted by \textrm{card}$(A)$ whereas the
dimension of subspace $W$ of some vector space will be written $\dim W.$

\section{The von-Neumann Formulas and the GKN Theorem\label{Section Two}}

Standard references for topics discussed in this section are \cite{Akhiezer
and Glazman, Dunford and Schwartz, Naimark, Reed and Simon, Rudin, Stone,
Weidmann, Zettl-book}.

Throughout this paper, the linear operator $T_{0}:\mathcal{D}(T_{0})\subseteq
H\rightarrow H$ will be an arbitrary closed, symmetric operator in $H$ while
$T_{1}:\mathcal{D}(T_{1})\subseteq H\rightarrow H$ is a linear operator
satisfying the operator inclusions
\begin{equation}
T_{1}^{\ast}=T_{0}\subseteq T_{0}^{\ast}:=T_{1}; \label{Property 1 of T}%
\end{equation}
in particular, we see that $T_{0}$ and $T_{1}$ are adjoints of each other.
Because of the inclusions in (\ref{Property 1 of T}), we call $T_{0}$ the
\textit{minimal }operator and $T_{1}$ the \textit{maximal }operator. Specific
reasons for this notation will be discussed below in this section (see also
Remark \ref{Remark 0}). Notice that if $T_{0}$ has a self-adjoint extension
$T$ in $H,$ then
\[
T_{0}\subseteq T=T^{\ast}\subseteq T_{0}^{\ast}=T_{1},
\]
so $T$ necessarily has the same form as $T_{1};$ that is,%
\[
Tx=T_{1}x\quad(x\in\mathcal{D}(T)).
\]
The general theory of self-adjoint extensions of the minimal operator $T_{0}$
(equivalently, self-adjoint restrictions of the maximal operator $T_{1})$ in a
Hilbert space - called the Stone-von Neumann theory - is discussed in depth in
\cite[Chapter XII, Section 4]{Dunford and Schwartz}. Of central importance in
this theory are two particular subspaces $X_{\pm}$ of $\mathcal{D}(T_{1})$,
defined by
\[
X_{\pm}:=\left\{  x\in\mathcal{D}(T_{1})\ \mid T_{1}x=\pm ix\right\}  ,
\]
where $i=\sqrt{-1}.$ These spaces are called the \textit{positive} and
\textit{negative deficiency spaces} of $T_{0}.$ The \textit{first von Neumann
formula} decomposes the maximal domain $\mathcal{D}(T_{1})$ into linearly
independent submanifolds:

\begin{theorem}
[The First von Neumann Formula]\label{First von Neumann Formula}%
$\mathcal{D}(T_{1})=\mathcal{D}(T_{0})+X_{+}+X_{-}.$
\end{theorem}

\noindent In fact, the sum in this formula is actually an orthogonal direct
sum. Indeed, under the graph inner product%
\begin{equation}
\left\langle x,y\right\rangle _{H}^{\ast}:=\left\langle x,y\right\rangle
_{H}+\left\langle T_{1}x,T_{1}y\right\rangle _{H}\quad(x,y\in\mathcal{D}%
(T_{1})) \label{Graph IP}%
\end{equation}
and associated norm%
\begin{equation}
\left(  \left\Vert x\right\Vert _{H}^{\ast}\right)  ^{2}=\left\Vert
x\right\Vert _{H}^{2}+\left\Vert T_{1}x\right\Vert _{H}^{2}\geq\left\Vert
x\right\Vert _{H}^{2}, \label{Graph Norm}%
\end{equation}
$\mathcal{D}(T_{1})$ is a Hilbert space and, with this inner product,
$\mathcal{D}(T_{0}),$ $X_{+}$ and $X_{-}$ are closed, orthogonal subspaces of
$\mathcal{D}(T_{1});$ see \cite[Chapter XII]{Dunford and Schwartz}. Notice
that if $x\in\mathcal{D}(T_{1})$ and
\[
x=x^{0}+x^{+}+x^{-},
\]
where $x^{0}\in\mathcal{D}(T_{0})$ and $x^{\pm}\in X_{\pm},$ then%
\begin{equation}
\left(  \left\Vert x\right\Vert _{H}^{\ast}\right)  ^{2}=\left(  \left\Vert
x^{0}\right\Vert _{H}^{\ast}\right)  ^{2}+\left(  \left\Vert x^{+}\right\Vert
_{H}^{\ast}\right)  ^{2}+\left(  \left\Vert x^{-}\right\Vert _{H}^{\ast
}\right)  ^{2}. \label{Orthogonality and Graph Norm}%
\end{equation}

\noindent The dimensions of $X_{\pm},$ denoted by $\mathrm{\dim}(X_{\pm}), $
are called the \textit{positive and negative deficiency indices} of $T_{0}$. A
key result in the Stone-von Neumann theory is that the equality of these
deficiency indices is equivalent to the existence of self-adjoint extensions
$T$ of $T_{0}$ in $H.$ Moreover, if \textrm{dim}$(X_{+})=$ \textrm{dim}%
$(X_{-})=0,$ $T_{0}=T_{1}$ is self-adjoint and is, in fact, the only
self-adjoint extension of $T_{0}$ in $H.$ In the case that \textrm{dim}%
$(X_{+})=$ \textrm{dim}$(X_{-}),$ we refer to this common value as the
\textit{deficiency index }and denote it by\textit{\ }$\mathrm{def}(T_{0}).$ In
addition to requiring the equality of these deficiency indices for the
entirety of this paper, we assume the deficiency indices are also finite.
Thus, another key assumption in this paper is:

\begin{condition}
\label{eq_def_ind} $1\leq\mathrm{def}(T_{0})$:$=$ $\mathrm{dim}(X_{+})=$
$\mathrm{dim}(X_{-})$ $<\infty.$
\end{condition}

The \textit{second von Neumann formula }gives a description of the domain of
any self-adjoint extension $T$ of $T_{0}$ in $H$:

\begin{theorem}
[The Second von Neumann Formula]Let $T:\mathcal{D}(T)\subseteq H\rightarrow H$
be a self-adjoint extension of $T_{0}.$ Then there exists an isometric
isomorphism $V:X_{+}\rightarrow X_{-}$ from the positive deficiency space
$X_{+}$ onto the negative deficiency space $X_{-}$ such that
\begin{align}
Tx  &  =T_{1}x\label{Form of SA}\\
\mathcal{D}(T)  &  =\{x+x_{+}+Vx_{+}\mid x\in\mathcal{D}(T_{0}),x_{+}\in
X_{+}\}. \label{2nd von Neumann formula}%
\end{align}
Conversely, if $T$ and its domain $\mathcal{D}(T)$ are defined through
$($\ref{Form of SA}$)$ and $($\ref{2nd von Neumann formula}$)$ for some
isometric isomorphism $V:X_{+}\rightarrow X_{-}$ , then $T$ is a self-adjoint
extension of $T_{0}.$
\end{theorem}

The Glazman-Krein-Naimark (GKN) theory is both a refinement and an application
of the Stone-von Neumann theory to self-adjoint operator extensions of
ordinary differential expressions. Excellent expositions of this theory can be
found in Akhiezer and Glazman \cite[Volume II, Chapter 8]{Akhiezer and
Glazman} and Naimark \cite[Part II, Chapter V]{Naimark}. To describe this
theory we assume, for the sake of simplicity, that $\ell\lbrack\cdot]$ is a
real, $2n$-th order Lagrangian symmetrizable differential expression of the
form
\begin{equation}
\ell\lbrack y](u)=\dfrac{1}{w(u)}\sum_{j=0}^{n}(-1)^{j}\left(  q_{j}%
(u)y^{(j)}(u)\right)  ^{(j)}\quad(u\in I), \label{ode}%
\end{equation}
where each coefficient $q_{j}:I\rightarrow\mathbb{R}$ in (\ref{ode}) is
$j$-times continuously differentiable on $I$ (noting, however, that general
`quasi-differentiable' conditions can be placed on these coefficients; see
also \cite{Zettl-RMJM}). The setting for the study of $\ell\lbrack\cdot]$ is
the Hilbert space
\[
L^{2}(I;w)=\{f:I\rightarrow\mathbb{C\mid}f\text{ is Lebesgue measurable and
}\int_{I}\left\vert f\right\vert ^{2}wdu<\infty\}
\]
endowed with the standard inner product%
\[
\left\langle f,g\right\rangle _{w}=\int_{a}^{b}f(u)\overline{g}(u)w(u)du\quad
(f,g\in L^{2}(I;w)).
\]
Here $I\subseteq\mathbb{R}$ is an open interval and $w$ is a positive (a.e.)
Lebesgue measurable function on $I$. The \textit{maximal operator}
$L_{1}:\mathcal{D}(L_{1})\subseteq L^{2}(I;w)\rightarrow L^{2}(I;w),$
generated by $\ell\lbrack\cdot],$ is defined to be
\begin{align*}
L_{1}f  &  =\ell\lbrack f]\\
f\in\mathcal{D}(L_{1})  &  =\{f:I\rightarrow\mathbb{C}\mid f^{(j)}\in
AC_{\mathrm{loc}}(I)\text{ }(j=0,1,\ldots,2n-1);\text{ }f,\ell\lbrack f]\in
L^{2}(I;w)\}.
\end{align*}
In this setting, the term `maximal' is appropriate; indeed, $\mathcal{D}%
(L_{1})$ - which is called the \textit{maximal domain} - is the largest
subspace of $L^{2}(I;w)$ for which the expression $\ell\lbrack\cdot]$ acts on
and maps into $L^{2}(I;w).$ It is clear that $L_{1}$ is a densely defined
operator. We denote the adjoint\ of $L_{1}$ by $L_{0};$ it is natural then to
call $L_{0}$ the \textit{minimal} \textit{operator} generated by $\ell
\lbrack\cdot].$ The GKN theory shows that, in fact, $L_{1}$ and $L_{0}$ are
adjoint to each other and $L_{0}$ is a closed symmetric operator in
$L^{2}(I;w).$ More explicitly, $L_{1}^{\ast}=L_{0}$ and
\[
L_{0}=\overline{L_{0}}\subseteq L_{0}^{\ast}=L_{1}.
\]

\begin{remark}
\label{Remark 0}The operators $T_{0}$ and $T_{1},$ defined earlier, are
analogous to the minimal operator $L_{0}$ and maximal operator $L_{1},$
respectively. Because of this, we call $T_{0}$ and $T_{1}$, respectively, the
minimal and maximal operators even though, in the general situation, the terms
maximal and minimal may not seem as appropriate as they do in the GKN theory.
Likewise, we shall call their respective domains the minimal domain
$\mathcal{D}(T_{0})$ and the maximal domain $\mathcal{D}(T_{1})$.
\end{remark}

The domain $\mathcal{D}(L_{0})$ of the minimal operator is given explicitly by%
\begin{equation}
\mathcal{D}(L_{0})=\{f\in\mathcal{D}(L_{1})\mid\left.  \lbrack f,g]_{w}%
\right\vert _{a}^{b}=0\text{ for all }g\in\mathcal{D}(L_{1})\},
\label{Minimal Domain}%
\end{equation}
where $\left.  [\cdot,\cdot]_{w}\right\vert _{a}^{b}$ is the skew-symmetric
bilinear form obtained from the classic Green's formula%
\begin{equation}
\left\langle L_{1}f,g\right\rangle _{w}-\left\langle f,L_{1}g\right\rangle
_{w}=\left.  [f,g]_{w}\right\vert _{a}^{b}\quad(f,g\in\mathcal{D}(L_{1})).
\label{Green's formula}%
\end{equation}
Moreover, we note that Condition (\ref{eq_def_ind}) is automatically satisfied
in this setting. Indeed, the deficiency indices of $L_{0}$ are equal since
$\ell$ has real coefficients and thus%
\[
\ell\lbrack f]=if\text{ \ }\Longleftrightarrow\overline{\ell\lbrack f]}%
=\ell\lbrack\overline{f}]=-i\overline{f}.
\]
Moreover, in this case,
\[
0\leq\mathrm{def}(L_{0})\leq2n.
\]
We are now in position to state the GKN Theorem. Notice that this theorem
provides a `recipe' for constructing all self-adjoint extensions of the
minimal operator $L_{0}$ in $L^{2}(I;w)$ by specifying certain restrictions
(boundary conditions), using the bilinear form $[\cdot,\cdot]_{w},$ on the
maximal domain $\mathcal{D}(L_{1}).$ We emphasize, however, that the original
GKN Theorem is valid \textit{only }for the minimal operator $L_{0}$ associated
with a real Lagrangian symmetrizable differential expressions of even order in
the specific Hilbert space $L^{2}(I;w)$. Compare the statement of the GKN
Theorem below with that of the GKN-EM Theorem (Theorem
\ref{The GKN-EM Theorem}) at the end of the next section.

\begin{theorem}
[The GKN Theorem]\label{The GKN Theorem}Suppose $L_{0}$ and $L_{1}$ are,
respectively, the minimal and maximal operators in $L^{2}(I;w),$ generated by
the differential expression $\ell\lbrack\cdot],$ given in $($\ref{ode}$)$. In
addition, let $m=$ \textrm{def}$(L_{0})$ so $0\leq m\leq2n.$

\begin{enumerate}
\item[(i)] Suppose the set $\{g_{j}\mid j=1,\ldots,m\}\subseteq\mathcal{D}%
(L_{1})$ satisfies the two conditions

\begin{enumerate}
\item[(a)]
\begin{equation}
\sum_{j=1}^{m}\alpha_{j}g_{j}\in\mathcal{D}(L_{0})\Longrightarrow\alpha
_{j}=0\text{ }(j=1,\ldots,m)\text{ and } \label{LI modulo minimal domain}%
\end{equation}

\item[(b)]
\begin{equation}
\left.  \lbrack g_{i},g_{j}]_{w}\right\vert _{a}^{b}=0\quad(i,j=1,\ldots
,\text{\textrm{def}}(L_{0})). \label{Glazman symmetry conditions}%
\end{equation}

\end{enumerate}

\noindent Define the operator $S:\mathcal{D}(S)\subseteq L^{2}(I;w)\rightarrow
L^{2}(I;w)$ by
\begin{align}
Sf  &  =L_{1}f\label{Operator S}\\
f\in\mathcal{D}(S)  &  =\{f\in\mathcal{D}(L_{1})\mid\left.  \lbrack
f,g_{j}]_{w}\right\vert _{a}^{b}=0\quad(j=1,\ldots,\text{\textrm{def}}%
(L_{0}))\}. \label{Domain of S}%
\end{align}
Then $S$ is a self-adjoint extension of the minimal operator $L_{0}$ in
$L^{2}(I;w).$

\item[(ii)] Conversely, if $S:\mathcal{D}(S)\subseteq L^{2}(I;w)\rightarrow
L^{2}(I;w)$ is a self-adjoint extension of the minimal operator $L_{0}$ in
$L^{2}(I;w),$ then there exists a set $\{g_{j}\mid j=1,\ldots,m\}\subseteq
\mathcal{D}(L_{1})$ satisfying the conditions $($%
\ref{LI modulo minimal domain}$)$ and $($\ref{Glazman symmetry conditions}$)$
such that $S$ is given explicitly by $($\ref{Operator S}$)\ $and
$($\ref{Domain of S}$).$
\end{enumerate}
\end{theorem}

\begin{remark}
\label{Remark -2}A collection of vectors $\{g_{j}\mid j=1,\ldots
,m\}\subseteq\mathcal{D}(L_{1})$ satisfying condition $($%
\ref{LI modulo minimal domain}$)$ are said to be linearly independent modulo
$\mathcal{D}(T_{0})$ while those that satisfy $($%
\ref{Glazman symmetry conditions}$)$ are said to satisfy Glazman symmetry
conditions. Further light, as well as a generalization, into these concepts be
made in the next section.
\end{remark}

\begin{remark}
\label{Remark -1}Each of the conditions
\[
\left.  \lbrack f,g_{j}]_{w}\right\vert _{a}^{b}=0\quad(j=1,\ldots
,\text{\textrm{def}}(L_{0})),
\]
given in $($\ref{Domain of S}$),$ is called a `boundary condition'. In the
case that \textrm{def}$(L_{0})=0,$ then $L_{1}$ $(=L_{0})$ is the only
self-adjoint extension of $L_{0}$ and, in this case, there are no boundary conditions.
\end{remark}

The GKN-EM Theorem, which we discuss in the next section in Theorem
\ref{The GKN-EM Theorem}, is a generalization of the GKN Theorem but,
remarkably, is valid in an \textit{arbitrary} Hilbert space for an
\textit{arbitrary} closed symmetric operator with equal and\textit{\ finite}
deficiency indices. This theorem is a highlight application of the general
complex symplectic theory developed by Everitt and Markus.

\section{Complex Symplectic Geometry and a Generalization of the GKN Theorem
\label{Section Two.Five}}

In a series of papers \cite{EM1, EM2, EM3, EM4}, Everitt and Markus developed
an extensive theory of complex symplectic geometry with applications to linear
ordinary and partial differential equations. Their work was motivated by their
interest in boundary value problems. In this section, we report on their
results that pertain to this manuscript. A highlight application of their
investigations is an important, and remarkable, generalization of Theorem
\ref{The GKN Theorem}; see Theorems \ref{The GKN-EM Theorem I} and
\ref{The GKN-EM Theorem} below. This generalization is key to the results we
establish in the next section.

\begin{definition}
\label{Complex Symplectic Space Definition}A complex symplectic space
\textsf{S }is a complex vector space together with a conjugate bilinear
$($sesquilinear$)$ complex-valued function $[\cdot:\cdot]:$ \textsf{S}$\times
$\textsf{S}$\rightarrow\mathbb{C}$ satisfying the properties

\begin{enumerate}
\item[(i)] $[c_{1}x_{1}+c_{2}x_{2}:y]=c_{1}[x_{1}:y]+c_{2}[x_{2}:y],$

\item[(ii)] $[x:y]=-\overline{[y:x]},$

\item[(iii)] $[x:$ \textsf{S}$]$ $=0$ $\Longrightarrow x=0$ $($non-degenerate
condition$)$.
\end{enumerate}

We call $[\cdot:\cdot]$ a $($non-degenerate$)$ symplectic form.
\end{definition}

Complex symplectic spaces are non-trivial generalizations (not merely
complexifications) of classical real symplectic spaces of Lagrangian and
Hamiltonian mechanics (see \cite{G-S}). Indeed, complex symplectic spaces have
a much wider scope and admit new applications. For example, whereas real
symplectic spaces cannot be odd dimensional, it is the case that, for every
$n\in\mathbb{N},$ there exists complex symplectic spaces of dimension $n.$

Along with their real symplectic counterparts, complex symplectic spaces
support the notion of \textit{Lagrangian} subspaces (see \cite{EM4} equation (1.10)).

\begin{definition}
\label{Complete Lagrangian Definition}A subspace \textsf{L} of a complex
symplectic space \textsf{S} is called Lagrangian if $[$\textsf{L} $:$
\textsf{L}$]=0;$ that is to say, when%
\[
\lbrack x:y]=0\quad(x,y\in\text{\textsf{L}}).
\]
A Lagrangian \textsf{L }$\subseteq$ \textsf{S }is called a complete Lagrangian
when
\[
x\in\text{\textsf{S} and }[x:\text{\textsf{L}}]=0\text{ }\Rightarrow\text{
}x\in\text{\textsf{L.}}%
\]

\end{definition}

We can characterize complete Lagrangian subspaces as follows. This
characterization is key for later results.

\begin{lemma}
\label{Characterization of Lagrangian subspaces}A Lagrangian subspace
\textsf{L }$\subseteq$ \textsf{S }is a complete Lagrangian if and only if
\begin{equation}
\text{\textsf{L }}=\{x\in\text{\textsf{S}}\mid\lbrack x:y]=0\text{ }%
(y\in\text{\textsf{L}})\}. \label{complete Lagrangian characterization}%
\end{equation}

\begin{proof}
Suppose \textsf{L }is a complete Lagrangian subspace of \textsf{S}. By
definition of complete, it is clear that $\{x\in$ \textsf{S}$\mid\lbrack
x:y]=0$ $(y\in$ \textsf{L}$)\}\subseteq$ \textsf{L}. On the other hand, if
$x\in$ \textsf{L }then $[x:y]=0$ for all $y\in$ \textsf{L} since \textsf{L }is
Lagrangian.\textsf{\ }Hence \textsf{L }$\subseteq\{x\in$ \textsf{S}%
$\mid\lbrack x:y]=0$ $(y\in$ \textsf{L}$)\}.$ Conversely, if \textsf{L }is
Lagrangian and given by $($\ref{complete Lagrangian characterization}$),$ then
it is clear that \textsf{L }is a complete.
\end{proof}
\end{lemma}

An essential step in the work of Everitt and Markus is a natural
generalization of the skew-symmetric bilinear form $\left.  [\cdot,\cdot
]_{w}\right\vert _{a}^{b}$ given by Green's formula (\ref{Green's formula}).

\begin{definition}
\label{symplectic_form} $[x,y]_{H}:=\left\langle T_{1}x,y\right\rangle
_{H}-\left\langle x,T_{1}y\right\rangle _{H}$ for $x,y\in\mathcal{D}%
(T_{1})\ .$
\end{definition}

Following the work of Everitt and Markus, we will see below that $[\cdot
,\cdot]_{H}$ can be identified with a degenerate symplectic form. We also note
that $[\cdot,\cdot]_{H}$ coincides with $\left.  [\cdot,\cdot]_{w}\right\vert
_{a}^{b}$ in the case $T_{1}$ is the maximal differential operator, generated
by $\ell\lbrack\cdot]$ (see (\ref{ode})), in the weighted Hilbert space
$L^{2}(I;w)$.

As shown in \cite{EM4}, the quotient space%
\begin{equation}
\text{\textsf{S}}^{\prime}:=\mathcal{D}(T_{1})/\mathcal{D}(T_{0}), \label{S'}%
\end{equation}
with zero element \textsf{0} $=$ $\mathcal{D}(T_{0}),$ is a complex symplectic
space when endowed with the form $[\cdot,\cdot]_{H}$; we outline the specific
details below.

Notice that, from Theorem \ref{First von Neumann Formula} and Condition
\ref{eq_def_ind}, \textsf{S}$^{^{\prime}}$ has dimension $2$\textrm{def}%
$(T_{0}).$ Indeed one may view \textsf{S}$^{^{\prime}}$ as an isomorphic copy
of the orthogonal sum of the deficiency spaces $X_{\pm}$ of $T_{0}.$ Everitt
and Markus call the space \textsf{S}$^{^{\prime}}$ the \textit{boundary space}
of $T_{0}.$ The elements of \textsf{S}$^{^{\prime}}$ are, of course, cosets
\textsf{x} $=$ $\{x+\mathcal{D}(T_{0})\}$ ($x\in\mathcal{D}(T_{1})$). In this
case, we call the vector $x$ a \textit{representative vector} of the coset
$\{x+\mathcal{D}(T_{0})\}.$

We now consider the natural projection $\phi:\mathcal{D}(T_{1})\rightarrow$
\textsf{S}$^{^{\prime}}$ defined by
\[
\phi(x)=\{x+\mathcal{D}(T_{0})\}\quad\left(  x\in\mathcal{D}(T_{1})\right)  .
\]
The following proposition makes clear the connection between a basis of a
subspace of \textsf{S}$^{^{\prime}}$ and the notion of linear independence
modulo $\mathcal{D}(T_{0})$ which we first encountered in Theorem
\ref{The GKN Theorem} and Remark \ref{Remark -2}.

\begin{lemma}
\label{cosets I}A collection of cosets $\{\phi t_{j}\}_{j=1}^{d},$ where
$\{t_{j}\}_{j=1}^{d}\subseteq\mathcal{D}(T_{1}),$ is a basis for a subspace of
dimension $d$ of the boundary space \textsf{S}$^{^{\prime}}$ if and only if
the representative vectors $\{t_{j}\}_{j=1}^{d}$ satisfy
\[
\sum_{j=1}^{d}\alpha_{j}t_{j}\in\mathcal{D}(T_{0})\Longrightarrow\alpha
_{j}=0\quad(j=1,2,\ldots,d);
\]
that is to say, $\{t_{j}\}_{j=1}^{d}$ is linearly independent modulo
$\mathcal{D}(T_{0}).$
\end{lemma}

\begin{proof}
The equation $\sum_{j=1}^{d}\alpha_{j}\phi t_{j}=$ \textsf{0} is equivalent to
$\sum_{j=1}^{d}\alpha_{j}t_{j}\in\mathcal{D}(T_{0}).$
\end{proof}

The following lemma generalizes the characterization of the domain of the
minimal operator; see (\ref{Minimal Domain}).

\begin{lemma}
\label{minimal domain} $\mathcal{D}(T_{0})=\left\{  x\in\mathcal{D}(T_{1}%
)\mid\lbrack x,y]_{H}=0\text{ }\left(  y\in\mathcal{D}(T_{1})\right)
\right\}  .$
\end{lemma}

\begin{proof}
Fix $x\in\mathcal{D}(T_{1})$ and suppose%
\[
\lbrack x,y]_{H}=0\quad(y\in\mathcal{D}(T_{1})).
\]
Since $[x,y]_{H}=-\overline{[y,x]}_{H},$ we see that $\left\langle
T_{1}y,x\right\rangle _{H}=\left\langle y,T_{1}x\right\rangle _{H}$ so
$x\in\mathcal{D}(T_{1}^{\ast})=\mathcal{D}(T_{0}).$ Conversely, let
$x\in\mathcal{D}(T_{0}).$ Since $T_{0}^{\ast}=T_{1}$ and $T_{0}x=T_{1}x,$ we
see that%
\[
\left\langle T_{1}x,y\right\rangle _{H}=\left\langle T_{0}x,y\right\rangle
_{H}=\left\langle x,T_{1}y\right\rangle _{H}\quad(y\in\mathcal{D}(T_{1}));
\]
that is, for each $y\in\mathcal{D}(T_{1}),$
\[
\lbrack x,y]_{H}=\left\langle T_{1}x,y\right\rangle _{H}-\left\langle
x,T_{1}y\right\rangle _{H}=0.
\]

\end{proof}

This result allows the boundary space \textsf{S}$^{^{\prime}}$ to be equipped
with a complex symplectic form.

\begin{definition}
[Boundary Space Symplectic Form]\label{symplectic form}
\begin{equation}
\lbrack\phi x:\phi y]_{\text{\textsf{S}}^{\prime}}:=[x,y]_{H}\quad
(x,y\in\mathcal{D}(T_{1})). \label{Symplectic Form Identification}%
\end{equation}

\end{definition}

Lemma \ref{minimal domain} assures Definition \ref{symplectic form} above is
independent of the choice of representative vectors. Moreover, Lemma
\ref{minimal domain} establishes the non-degeneracy property of Definition
\ref{Complex Symplectic Space Definition}. From the definition of a Lagrangian
subspace, the following extension of Lemma \ref{cosets I} is clear.

\begin{proposition}
\label{Lagrangians I}A collection of cosets $\{\phi t_{j}\}_{j=1}^{d}$ form a
basis for a $d$-dimensional Lagrangian subspace of the boundary space
\textsf{S}$^{\prime}$ if and only if the representative vectors $\{t_{j}%
\}_{j=1}^{d}$ satisfy

\begin{enumerate}
\item[(a)]
\begin{equation}
\sum_{j=1}^{d}\alpha_{j}t_{j}\in\mathcal{D}(T_{0})\Longrightarrow\alpha
_{j}=0\text{ }(j=1,\ldots,d);\text{ } \label{LI modulo minimal domain III}%
\end{equation}
and

\item[(b)]
\begin{equation}
\left[  t_{i},t_{j}\right]  _{H}=0\quad(i,j=1,\ldots,d).
\label{Lagrangian  conditions}%
\end{equation}

\end{enumerate}
\end{proposition}

Notice that the properties (\ref{LI modulo minimal domain III}) and
(\ref{Lagrangian conditions}) are identical to those conditions discussed in
Theorem \ref{The GKN Theorem}. Because of their importance in the special case
when $d=\mathrm{def}\left(  T_{0}\right)  ,$ we incorporate these two
properties into the following definition.

\begin{definition}
\label{glazman_set} A collection of vectors $\left\{  t_{j}\mid j=1,\ldots
,\mathrm{def}\left(  T_{0}\right)  \right\}  \subseteq\mathcal{D}(T_{1})$ is
called a GKN set for $T_{0}$ if

\begin{enumerate}
\item[(i)] the set $\left\{  t_{j}\mid j=1,\ldots,\mathrm{def}\left(
T_{0}\right)  \right\}  $ is linearly independent modulo the minimal domain
$\mathcal{D}(T_{0});$ that is to say
\begin{equation}
\text{if }\sum_{j=1}^{\mathrm{def}\left(  T_{0}\right)  }\alpha_{j}t_{j}%
\in\mathcal{D}(T_{0})\text{ then }\alpha_{j}=0\text{ for }j=1,\ldots
,\mathrm{def}\left(  T_{0}\right)  ; \label{LI mod T_0}%
\end{equation}
and

\item[(ii)] the set $\left\{  t_{j}\mid j=1,\ldots,\mathrm{def}\left(
T_{0}\right)  \right\}  $ satisfies the symmetry conditions
\begin{equation}
\lbrack t_{i},t_{j}]_{H}=0\qquad(i,j=1,\ldots,\mathrm{def}\left(
T_{0}\right)  ). \label{gkn_sym_cond1}%
\end{equation}

\end{enumerate}
\end{definition}

\begin{remark}
\label{Remark 2}Observe that if $G$ $\subseteq$ $\mathcal{D}(T_{1})$ is a GKN
set for $T_{0},$ then any non-empty, proper subset $P\subseteq G$ is linearly
independent modulo $\mathcal{D}(T_{0})$ and satisfies the symmetry conditions
in $($\ref{gkn_sym_cond1}$).$ We refer to $P$ as a partial GKN\ set. However,
we note that the only partial GKN sets $P$ that we use in this manuscript are
those which satisfy \textrm{card}$(P)=\mathrm{dim}(W)\leq$ $\mathrm{def}%
(T_{0}),$ where $W$ is a complex finite-dimensional extension space; see
Condition \ref{Dimensionality Condition} in Section \ref{Section Three}.
\end{remark}

We now turn our attention to characterizing complete Lagrangians. A key result
of Everitt and Markus in this setting is that not only do complete Lagrangians
\textsf{L }exist (see \cite[Equations (1.54) and (1.61)]{EM4}) but their
dimensions are precisely that of the deficiency index; that is,
\begin{equation}
\mathrm{dim\ }\text{\textsf{L}}=\mathrm{def}(T_{0})
\label{dimension of complete Lagrangians}%
\end{equation}
(see \cite[Equation (3.9)]{EM4}). Moreover,

\begin{lemma}
\label{New Characterization of complete Lagrangians}With \textrm{def}%
$(T_{0})<\infty,$ a Lagrangian subspace \textsf{L }$\subseteq$ \textsf{S}%
$^{\prime}$ is complete if and only if each of the two conditions hold:

\begin{enumerate}
\item[(i)] \textrm{dim \textsf{L }}$=$ \textrm{\ def}$(T_{0});$

\item[(ii)] \textsf{L }$=\{\phi x\mid\lbrack\phi x:\phi t_{j}%
]_{\text{\textsf{S}}^{\prime}}=0$ $(j=1,2,\ldots,$\textrm{def}$(T_{0}))\}$ for
some GKN set $\{t_{j}\mid j=1,2,\ldots,$\textrm{def}$(T_{0})\}.$ \newline
Moreover, in this case,%
\begin{equation}
\phi^{-1}\text{\textsf{L }}=\{x\in\mathcal{D}(T_{1})\mid\lbrack x,t_{j}%
]_{H}=0\text{ }(j=1,2,\ldots,\mathrm{def}(T_{0}))\}. \label{Lemma New 1}%
\end{equation}

\end{enumerate}

\begin{proof}
Suppose \textsf{L }$\subseteq$ \textsf{S}$^{^{\prime}}$ is complete. Then, by
$($\ref{dimension of complete Lagrangians}$),$ \textrm{dim \textsf{L }}$=$
\textrm{\ def}$(T_{0})$ establishing $($i$).$ By Lemma
\ref{Characterization of Lagrangian subspaces},
\begin{equation}
\text{\textsf{L}}=\{\phi x\mid\lbrack\phi x:\phi y]_{\text{\textsf{S}}%
^{\prime}}=0\text{ }(\phi y\in\text{\textsf{L}})\}. \label{Lemma New 2}%
\end{equation}
Let $\{\phi t_{j}\mid j=1,2,\ldots,$\textrm{def}$(T_{0})\}$ be a basis for
\textsf{L}. Then, by Proposition \ref{Lagrangians I}, $\{t_{j}\mid
j=1,2,\ldots,$\textrm{def}$(T_{0})\}$ is a GKN set for $T_{0}.$ It follows
from $($\ref{Lemma New 2}$),$ that%
\begin{equation}
\text{\textsf{L}}=\{\phi x\mid\lbrack\phi x:\phi t_{j}]_{\text{\textsf{S}%
}^{\prime}}=0\text{ }(j=1,2,\ldots,\mathrm{def}(T_{0}))\}, \label{Lemma New 3}%
\end{equation}
proving $($ii$).$ Lastly, using the identification in $($%
\ref{Symplectic Form Identification}$)$ along with the identity in
$($\ref{Lemma New 3}$),$ $($\ref{Lemma New 1}$)$ is clear$.$ \medskip\newline
Conversely, suppose $($i$)$ and $($ii$)$ hold. It is straightforward to show
that \textsf{L }is a subspace of \textsf{S}$^{^{\prime}}.$ Clearly
$($\ref{Lemma New 1}$)$ follows from $($ii$).$ Moreover, since $\{t_{j}\mid
j=1,2,\ldots,$\textrm{def}$(T_{0})\}$ is a GKN set for $T_{0},$ we see that%
\[
\lbrack\phi t_{i}:\phi t_{j}]_{\text{\textsf{S}}^{\prime}}=[t_{i},t_{j}%
]_{H}=0.
\]
It follows by taking linear combinations that \textsf{L }is Lagrangian.
Finally, from $($\ref{dimension of complete Lagrangians}$),$ we see that
\textsf{L }is complete.
\end{proof}
\end{lemma}

The authors in \cite[Theorem 1.14 and Remark 1.15]{EM4} establish the
following characterization of self-adjoint extensions of $T_{0}$ in terms of
complete Lagrangian subspaces \textsf{L }of \textsf{S}$^{^{\prime}}.$

\begin{theorem}
[The Finite-Dimensional GKN-EM Theorem]\label{The GKN-EM Theorem I}Let $T_{0}$
and $T_{1}$ be, respectively, the minimal and maximal operators as defined in
Section \ref{Section Two} and let \textsf{S}$^{^{\prime}}$ be given by
$($\ref{S'}$)$. There exists a one-to-one correspondence between the set
$\{T\}$ of all self-adjoint extensions of $T_{0}$ and the set $\{$%
\textsf{L}$\}$ of all complete Lagrangians \textsf{L }$\subseteq$
\textsf{S}$^{^{\prime}}.$ More specifically,

\begin{enumerate}
\item[(a)] if $T$ is a self-adjoint operator with $T_{0}\subseteq T\subseteq
T_{1}$, then
\begin{equation}
\text{\textsf{L}}:=\{\phi x\in\text{\textsf{S}}^{^{\prime}}\mid x\in
\mathcal{D}(T)\} \label{GKN-EM I - 1}%
\end{equation}
is a complete Lagrangian subspace of \textsf{S}$^{^{\prime}}$ of dimension
\textrm{def}$(T_{0}).$ Moreover, $\phi^{-1}$\textsf{L }$=\mathcal{D}(T).$

\item[(b)] If \textsf{L }is a complete Lagrangian subspace of \textsf{S}%
$^{^{\prime}},$ then \textsf{L }has dimension \textrm{def}$(T_{0}).$ Define%
\[
\mathcal{D}(T)=\{x\in\mathcal{D}(T_{1})\mid\phi x\in\text{ \textsf{L}}\}.
\]
Then $T:\mathcal{D}(T)\subset H\rightarrow H$ given by
\begin{align*}
Tx  &  =T_{1}x\\
x  &  \in\mathcal{D}(T)
\end{align*}
is a self-adjoint operator satisfying $T_{0}\subseteq T\subseteq T_{1}.$
Moreover, $\phi^{-1}$ \textsf{L} $=\mathcal{D}(T).$
\end{enumerate}
\end{theorem}

Combining Theorem \ref{The GKN-EM Theorem I} with Lemmas
\ref{Characterization of Lagrangian subspaces} and
\ref{New Characterization of complete Lagrangians}, we are now in position to
state and prove an important consequence of Theorem \ref{The GKN-EM Theorem I}
which, for our purposes, is key to the results developed in the next section
and in the examples of Section \ref{Examples}. We note that the next theorem
is an \textit{exact} generalization of the GKN theorem stated in Theorem
\ref{The GKN Theorem}.

\begin{theorem}
[The Finite-Dimensional Symplectic GKN-EM Theorem]\label{The GKN-EM Theorem}%
Suppose $T_{0}$ and $T_{1}$ are linear operators satisfying the conditions set
forth in Section \ref{Section Two} and $[\cdot,\cdot]_{H}$ is the symplectic
form defined in Definition \ref{symplectic_form}. In particular, we assume
$T_{0}$ has equal and finite deficiency indices denoted by $\mathrm{def}%
\left(  T_{0}\right)  $.

\begin{enumerate}
\item[(i)] If the operator $T:\mathcal{D}(T)$ $\subseteq H\rightarrow H$ is
self-adjoint and satisfies
\[
T_{0}\subseteq T\subseteq T_{1}%
\]
then there exists a GKN set $\left\{  t_{j}\mid j=1,\ldots,\mathrm{def}\left(
T_{0}\right)  \right\}  \subseteq\mathcal{D}(T_{1})$ of $T_{0}$ such that
\begin{equation}
\mathcal{D}(T)=\{x\in\mathcal{D}(T_{1})\ \mid\lbrack x,t_{j}]_{H}=0\text{
}(j=1,\ldots,\mathrm{def}\left(  T_{0}\right)  )\}. \label{3.2.0}%
\end{equation}

\item[(ii)] If $\left\{  t_{j}\mid j=1,\ldots,\mathrm{def}\left(
T_{0}\right)  \right\}  \subseteq\mathcal{D}(T_{1})$ is a GKN set for $T_{0}$
then the operator $T:\mathcal{D}(T)$ $\subseteq H\rightarrow H$ given by
\begin{align}
Tx  &  =T_{1}x\label{gknem1}\\
x\in\mathcal{D}(T)  &  =\{x\in\mathcal{D}(T_{1})\mid\lbrack x,t_{j}%
]_{H}=0\text{ }(j=1,\ldots,\text{\textrm{def}}(T_{0}))\} \label{gknem2}%
\end{align}
is self-adjoint and satisfies
\[
T_{0}\subseteq T\subseteq T_{1}.
\]

\end{enumerate}

\begin{proof}
\underline{$($i$)$} Suppose $T:\mathcal{D}(T)$ $\subseteq H\rightarrow H$ is
self-adjoint and satisfies $T_{0}\subseteq T\subseteq T_{1}$. By Theorem
\ref{The GKN-EM Theorem I}$,$
\begin{equation}
\text{\textsf{L}}=\{\phi x\in\text{\textsf{S}}^{^{\prime}}\mid x\in
\mathcal{D}(T)\} \label{3.2.1}%
\end{equation}
is a complete Lagrangian subspace of \textsf{S}$^{^{\prime}}$ of dimension
$\mathrm{def}\left(  T_{0}\right)  $ from which it follows that
\begin{equation}
\phi^{-1}\text{\textsf{L }}=\mathcal{D}(T). \label{3.2.2}%
\end{equation}
Moreover, by Lemma \ref{New Characterization of complete Lagrangians}, there
exists a GKN set $\{t_{j}\mid j=1,2,\ldots,$\textrm{def}$(T_{0})\}$ for
$T_{0}$ such that%
\[
\text{\textsf{L}}=\{\phi x\mid\lbrack\phi x:\phi t_{j}]_{\text{\textsf{S}%
}^{\prime}}=0\text{ }(j=1,2,\ldots,\mathrm{def}(T_{0}))\}.
\]
and
\begin{equation}
\phi^{-1}\text{\textsf{L }}=\{x\in\mathcal{D}(T_{1})\mid\lbrack x,t_{j}%
]_{H}=0\text{ }(j=1,2,\ldots,\mathrm{def}(T_{0}))\}. \label{3.2.3}%
\end{equation}
Comparing $($\ref{3.2.2}$)$ with $($\ref{3.2.3}$),$ we obtain $($%
\ref{3.2.0}$)$. \medskip\newline\underline{$($ii$)$} Suppose $\{t_{j}\mid
j=1,2,\ldots,$\textrm{def}$(T_{0})\}$ is a GKN set for $T_{0}.$ Let%
\begin{equation}
\text{\textsf{L}}=\{\phi x\mid\lbrack\phi x:\phi t_{j}]_{\text{\textsf{S}%
}^{\prime}}=0\text{ }(j=1,2,\ldots,\mathrm{def}(T_{0}))\}. \label{3.2.4}%
\end{equation}
By Lemma \ref{New Characterization of complete Lagrangians}, \textsf{L} is a
complete Lagrangian subspace of \textsf{S}$^{^{\prime}}$ of dimension
$\mathrm{def}(T_{0}).$ Define $T$ as in $($\ref{gknem1}$)$ and $($%
\ref{gknem2}$).$ Then, from $($\ref{gknem2}$)$ and $($\ref{3.2.4}$),$ we see
that%
\[
\text{\textsf{L}}=\{\phi x\mid x\in\mathcal{D}(T)\}
\]
so that%
\[
\mathcal{D}(T)=\phi^{-1}\text{\textsf{L}}=\{x\in\mathcal{D}(T_{1})\mid\phi
x\in\text{\textsf{L}}\}.
\]
By Theorem \ref{The GKN-EM Theorem I}, $T$ is self-adjoint and $T_{0}\subseteq
T\subseteq T_{1}.$
\end{proof}
\end{theorem}

\begin{remark}
\label{Remark 3}In the case that $H=L^{2}(I;w)$ and $T_{0}$ and $T_{1}$ are,
respectively, the minimal and maximal operators $L_{0}$ and $L_{1},$ generated
by the ordinary differential expression $($\ref{ode}$),$ Theorem
\ref{The GKN-EM Theorem} is \underline{identical} to the classical GKN theorem
given in Theorem \ref{The GKN Theorem}. Again, it is remarkable that the GKN
theorem extends verbatim to a general Hilbert space with an arbitrary closed
symmetric operator having equal deficiency indices. As in the classical GKN
setting, we also call the conditions
\[
\lbrack x,t_{j}]_{H}=0\text{ for }j=1,\ldots,\mathrm{def}\left(  T_{0}\right)
\]
`boundary conditions'. Lastly, we note that, as in Remark \ref{Remark -1}, if
\textrm{def}$(T_{0})=0,$ there are no such boundary conditions and, in this
case, the only self-adjoint extension of $T_{0}$ is the maximal operator
$T_{1}$ $(=T_{0}).$
\end{remark}

\begin{remark}
Everitt and Markus discuss other important applications of their results to
ordinary and partial differential operators. We refer the reader to Sections
2.1, 2.2 and 4.2 in \cite{EM4}. They outline the argument given above in
Theorem \ref{The GKN-EM Theorem} for Sturm-Liouville problems $($see
\cite[Section 2, equations $($2.23$),$ $($2.24$),$ and $($2.25$)$]{EM4}$)$ as
well as for general Shin-Zettl quasi-differential operators $($see
\cite[Section 4; in particular equations $($4.57$)$--$($4.61$)$]{EM4}$)$. The
authors are certain that the most general situation $($when $T_{0}$ has finite
and equal deficiency indices$)$, which we prove in Theorem
\ref{The GKN-EM Theorem}, was known to Everitt and Markus but we cannot find
an exact reference in their joint work.
\end{remark}

\section{Maximal and Minimal Operators in $H\oplus W$\label{Section Three}}

We remind the reader that $T_{0}:\mathcal{D}(T_{0})\subseteq H\rightarrow H$
is a closed, symmetric operator with equal, finite deficiency indices
\textrm{def}$(T_{0})$ and adjoint operator $T_{1}$ satisfying $T_{1}^{\ast
}=T_{0}\subseteq T_{0}^{\ast}=T_{1}.$ In this section, we identify a
\textit{family} of minimal operators $\widehat{T}_{0}:\mathcal{D(}\widehat
{T}_{0})\subseteq H\oplus W\rightarrow H\oplus W$ and an associated family of
\textit{maximal} operators $\widehat{T}_{1}:\mathcal{D(}\widehat{T}%
_{1})\subseteq H\oplus W\rightarrow H\oplus W$ in the extended space $H\oplus
W$ generated by, respectively, the minimal operator $T_{0}$ and the maximal
operator $T_{1}$ in the base space $H.$ We show that each $\widehat{T}_{0}$ is
a closed, symmetric operator in $H\oplus W$ with equal deficiency indices and
\textrm{def}$(\widehat{T}_{0})=\mathrm{def}(T_{0})$. Moreover, the operators
$\widehat{T}_{0}$ and $\widehat{T}_{1}$ are adjoints of each other just as in
the classical case with $T_{0}$ and $T_{1}.$

A fundamental assumption in our development of the maximal and minimal
operators in $H\oplus W$ is the following dimensionality requirement for the
extension space:

\begin{condition}
\label{Dimensionality Condition}$\dim(W)\leqslant\mathrm{def}\left(
T_{0}\right)  .$
\end{condition}

Fix a partial GKN set
\begin{equation}
\{t_{j}\mid j=1,\ldots,\dim(W)\}\subseteq\mathcal{D}(T_{1}); \label{gkn}%
\end{equation}
recall, from Remark \ref{Remark 2} and Condition
\ref{Dimensionality Condition}, that this set exists and satisfies the two
conditions
\begin{equation}
\sum_{j=1}^{\dim(W)}\alpha_{j}t_{j}\in\mathcal{D}(T_{0})\Longrightarrow
\alpha_{j}=0\qquad(j=1,\ldots,\dim(W)) \label{linear independence}%
\end{equation}
and
\begin{equation}
\lbrack t_{i},t_{j}]_{H}=0\qquad(i,j=1,\ldots,\dim(W)). \label{gkn symmetry}%
\end{equation}
It is clear that the maximal operator $T_{1}$ in the base space is symmetric
on
\begin{equation}
\Delta_{0}:=\mathcal{D}(T_{0})+\mathrm{span}\{t_{j}\mid j=1,\ldots
,\dim(W)\subseteq\mathcal{D}(T_{1}). \label{Delta_0 set}%
\end{equation}
Now let
\[
\left\{  \xi_{j}\mid j=1,\ldots,\dim(W)\right\}  \subseteq W
\]
be an orthonormal basis of $W$ and define $\Psi:\Delta_{0}\rightarrow W$ by
\begin{align}
\Psi\left(  t_{j}\right)   &  =\xi_{j}\mathit{\ }\text{\textit{\quad}%
}(j=1,\ldots,\dim(W))\nonumber\\
\Psi\left(  s\right)   &  =0\mathit{\ \quad}(s\in\mathcal{D}(T_{0})).
\label{Psi zero on minimal domain}%
\end{align}
and extend $\Psi$ to $\Delta_{0}$; that is to say
\[
\Psi\left(  s+\sum_{j=1}^{\dim(W)}\alpha_{j}t_{j}\right)  =\sum_{j=1}%
^{\dim(W)}\alpha_{j}\xi_{j}.
\]
Note the key fact that $\Psi$ maps the partial GKN set $\{t_{j}\mid
j=1,\ldots,\dim(W)\}$ onto $W.$

Lastly, fix an arbitrary self-adjoint operator $B:W\rightarrow W$ in the
extension space $W$. With these definitions and conditions in place, we are
now in position to define a \textit{minimal operator} $\widehat{T_{0}}$ in
$H\oplus W$ generated by $T_{0}$.

\begin{definition}
\label{min_op}The minimal operator $\widehat{T}_{0}:\mathcal{D}\left(
\widehat{T}_{0}\right)  \subseteq H\oplus W\rightarrow H\oplus W$ is defined
to be%
\begin{align}
\widehat{T}_{0}(x,a)  &  =(T_{1}x,Ba)\label{T_0 form}\\
(x,a)\in\mathcal{D(}\widehat{T}_{0})  &  :=\left\{  (x,\Psi x)\mid x\in
\Delta_{0}\right\}  . \label{T_0 domain}%
\end{align}

\end{definition}

At this point, it is unclear why we call $\widehat{T}_{0}$ the
\textit{minimal} operator generated by $T_{0};$ we will justify this
terminology in Remark \ref{Remark 4}. In Theorem\textit{\ }\ref{min_op_closed}
below we show that the minimal operator $\widehat{T}_{0}$ is, in fact, a
densely defined operator which is both closed and symmetric. Moreover, in
Theorem \ref{adjoint_lemma}, where it is shown that $(\widehat{T}_{0})^{\ast
}=\widehat{T}_{1},$ we introduce the important linear transformation
$\Omega:\mathcal{D}(T_{1})\ \rightarrow W$ defined by
\begin{equation}
\Omega x:=\sum_{j=1}^{\dim(W)}[x,t_{j}]_{H}\xi_{j}\quad(x\in\mathcal{D}%
(T_{1})). \label{Omega map}%
\end{equation}

Observe, by definition of the partial GKN set $\{t_{j}\mid j=1,\ldots,\dim
W\}$ and Lemma \ref{minimal domain}, that
\begin{equation}
\Omega x=0\quad(x\in\Delta_{0}). \label{Property of Omega}%
\end{equation}
With this transformation $\Omega,$ we are now ready to introduce the maximal
operator $\widehat{T}_{1}$.

\begin{definition}
\label{max_op} The maximal operator $\widehat{T}_{1}:\mathcal{D}\left(
\widehat{T}_{1}\right)  \subseteq H\oplus W\rightarrow H\oplus W$ is defined
by%
\begin{align}
\widehat{T}_{1}(x,a)  &  =(T_{1}x,Ba-\Omega x)\label{T_1 form}\\
(x,a)\in\mathcal{D(}\widehat{T}_{1})  &  :=\left\{  (x,a)\mid x\in
\mathcal{D}(T_{1}),a\in W\right\}  . \label{T_1 domain}%
\end{align}

\end{definition}

Note that if $(x,\Psi x)\in\mathcal{D}(\widehat{T}_{0}),$ then $(x,\Psi
x)\in\mathcal{D}(\widehat{T}_{1}).$ Moreover, in this case, $\Omega x=0$ by
(\ref{Property of Omega}) so
\[
\widehat{T}_{1}(x,\Psi x)=(T_{1}x,B\Psi x-\Omega x)=(T_{1}x,B\Psi
x)=\widehat{T}_{0}(x,\Psi x);
\]
that is%
\begin{equation}
\widehat{T}_{0}\subseteq\widehat{T}_{1}. \label{T_0 contained in T_1}%
\end{equation}

\begin{remark}
\label{Remark 4}The term `maximal' is appropriate; indeed, observe that
$\mathcal{D(}\widehat{T}_{1})$ is the largest linear manifold in $H\oplus W$
on which an operator representation of $T_{1}$ makes sense. Moreover, once we
establish the fact that $(\widehat{T}_{0})^{\ast}=\widehat{T}_{1},$ we see
that the term `minimal' is appropriate for the operator $\widehat{T}_{0}. $
\end{remark}

\begin{proposition}
\label{closure_prop} The extension $J:\mathcal{D}(J)\subseteq H\rightarrow H$
of the \textit{minimal operator} $T_{0},$ defined by%
\begin{align*}
Jx  &  :=T_{1}x\\
x\in\mathcal{D}(J)  &  :=\Delta_{0},
\end{align*}
is a closed symmetric operator.
\end{proposition}

\begin{proof}
Since $T_{0}$ is densely defined and $\mathcal{D}(T_{0})\subseteq\Delta_{0},$
it is clear that $\mathcal{D}(J)$ is dense in $H.$ Now, from Lemma
\ref{minimal domain} and $($\ref{gkn_sym_cond1}$),$ we see that
\[
\lbrack x,y]_{H}=0\quad(x,y\in\Delta_{0}\subseteq\mathcal{D}(T_{1})).
\]
Hence, from Definition \ref{symplectic_form},
\[
0=[x,y]_{H}=\left\langle T_{1}x,y\right\rangle _{H}-\left\langle
x,T_{1}y\right\rangle _{H}=\left\langle Jx,y\right\rangle _{H}-\left\langle
x,Jy\right\rangle _{H}\quad(x,y\in\mathcal{D}(J)),
\]
establishing that $J$ is symmetric in $H.$ To show that $J$ is closed, suppose
first that $\mathrm{dim}(\Delta_{0})$ $(\mathrm{mod}(\mathcal{D}\left(
T_{0}\right)  ))=1$; that is%
\begin{equation}
\mathcal{D}(J)=\mathcal{D}(T_{0})+\text{\textrm{span}}\{t_{1}\}, \label{J0}%
\end{equation}
where $t_{1}\in\mathcal{D}(T_{1})\diagdown\mathcal{D}(T_{0}).$ Consider a
sequence $\left\{  x_{j}\right\}  \subseteq\Delta_{0}$ and vectors $x,y\in H$
such that $x_{j}\rightarrow x$ and $T_{1}x_{j}=Jx_{j}\rightarrow y$ where the
convergence of both sequences is in $H.$ Of course, we need to show
\begin{equation}
x\in\mathcal{D}(J) \label{J1}%
\end{equation}
and
\begin{equation}
Jx=y. \label{J2}%
\end{equation}
Since $\mathcal{D}(J)\subseteq\mathcal{D}(T_{1})$ and $T_{1}$ is closed, we
know that $x\in\mathcal{D}(T_{1})$ and $T_{1}x=y.$ Hence $($\ref{J2}$)$ will
be established once we show $($\ref{J1}$).$ Now, by $($\ref{J0}$)$ and Theorem
\ref{First von Neumann Formula}, we can write%
\begin{align}
x_{j}  &  =v_{j}+\alpha_{j}t_{1}\label{J3}\\
&  =(\alpha_{j}t_{1}^{0}+v_{j})+\alpha_{j}t_{1}^{+}+\alpha_{j}t_{1}^{-}%
\quad(j\in\mathbb{N})\nonumber
\end{align}
where $t_{1}=t_{1}^{0}+t_{1}^{+}+t_{1}^{-}$ and $v_{j},$ $t_{1}^{0}$
$\in\mathcal{D}(T_{0}),$ $t_{1}^{\pm}\in X_{\pm}$ and $x=x^{0}+x^{+}+x^{-},$
where $x^{0}\in\mathcal{D}(T_{0})$ and $x^{\pm}\in X_{\pm}.$ Since
$x_{j}\rightarrow x$ and $T_{1}x_{j}\rightarrow T_{1}x,$ we see that%
\[
\left\Vert x_{j}-x\right\Vert _{H}^{\ast}\rightarrow0\text{ as }%
j\rightarrow\infty,
\]
where $\left\Vert \cdot\right\Vert _{H}^{\ast}$ is the graph norm given in
$($\ref{Graph Norm}$)$. Since%
\[
\left(  \left\Vert x_{j}-x\right\Vert _{H}^{\ast}\right)  ^{2}=\left(
\left\Vert \alpha_{j}t_{1}^{0}+v_{j}-x^{0}\right\Vert _{H}^{\ast}\right)
^{2}+\left(  \left\Vert \alpha_{j}t_{1}^{+}-x^{+}\right\Vert _{H}^{\ast
}\right)  ^{2}+\left(  \left\Vert \alpha_{j}t_{1}^{-}-x^{-}\right\Vert
_{H}^{\ast}\right)  ^{2},
\]
we see, from $($\ref{Graph Norm}$)$ and $($\ref{Orthogonality and Graph Norm}%
$)$, that
\begin{align}
\alpha_{j}t_{1}^{0}+v_{j}  &  \rightarrow x^{0}\text{ in }H\label{J4}\\
\alpha_{j}t_{1}^{+}  &  \rightarrow x^{+}\text{ in }H\label{J5}\\
\alpha_{j}t_{1}^{-}  &  \rightarrow x^{-}\text{ in }H. \label{J6}%
\end{align}
Since $t_{1}^{+}$ and $t_{1}^{-}$ both cannot be zero (otherwise, $t_{1}%
=t_{1}^{0}\in\mathcal{D}(T_{0})$, contradicting our choice of $t_{1}),$ we see
from either $($\ref{J5}$)$ or $($\ref{J6}$)$ that there exists $\alpha
\in\mathbb{C}$ with $\alpha_{j}\rightarrow\alpha$. It follows that $\alpha
_{j}t_{1}\rightarrow\alpha t_{1}$ in $H.$ Then, from $($\ref{J3}$)$ and
$($\ref{J4}$),$ we see that%
\begin{align*}
x_{j}  &  =v_{j}+\alpha_{j}t_{1}\\
&  =\alpha_{j}t_{1}^{0}+v_{j}+\alpha_{j}t_{1}-\alpha_{j}t_{1}^{0}\\
&  \rightarrow(x^{0}-\alpha t_{1}^{0})+\alpha t_{1}\in\mathcal{D}%
(T_{0})+\mathrm{span}\{t_{1}\}.
\end{align*}
Hence we see that $x=(x^{0}-\alpha t_{1}^{0})+\alpha t_{1}\in\mathcal{D}(J), $
as required. The general proof of this proposition follows by induction on
\textrm{dim}$(W).$
\end{proof}

\begin{remark}
\label{Remark 5}Proposition \ref{closure_prop} shows that, on $\Delta_{0},$
the maximal operator $T_{1}$ is a closed, symmetric operator. Of course,
$T_{1}$ is not, in general, symmetric on $\mathcal{D}(T_{1}).$
\end{remark}

\begin{theorem}
\label{min_op_closed} The operator $\widehat{T}_{0}$ is a closed, densely
defined symmetric operator in $H\oplus W$.
\end{theorem}

\begin{proof}
(i) \underline{ $\widehat{T}_{0}$ is Hermitian\newline}: \newline Let $(x,\Psi
x),$ $(y,\Psi y)\in\mathcal{D}(\widehat{T}_{0}).$ Then, by Proposition
\ref{closure_prop} and the fact that $B$ is symmetric in $W$, we see that
\begin{align*}
\left\langle \widehat{T}_{0}(x,\Psi x),(y,\Psi y)\right\rangle _{H\oplus W}
&  =\left\langle (T_{1}x,B\Psi x),(y,\Psi y)\right\rangle _{H\oplus W}\\
&  =\left\langle T_{1}x,y\right\rangle _{H}+\left\langle B\Psi x,\Psi
y\right\rangle _{W}\\
&  =\left\langle x,T_{1}y\right\rangle _{H}+\left\langle \Psi x,B\Psi
y\right\rangle _{W}\\
&  =\left\langle (x,\Psi x),\widehat{T}_{0}(y,\Psi y)\right\rangle _{H\oplus
W}.
\end{align*}
Hence $\widehat{T}_{0}$ is Hermitian.\medskip

(ii) \underline{$\mathcal{D}(\widehat{T}_{0})$ is dense in $H\oplus W$%
:}\newline Since $\mathcal{D}(T_{0})$ is dense in $H$ and $\Psi$ is
surjective, it is clear that $\mathcal{D}(\widehat{T}_{0})$ is dense in
$H\oplus W.\medskip$

(iii) \underline{$\widehat{T}_{0}$ is symmetric in $H\oplus W$}:

\noindent This follows immediately from (i) and (ii)\medskip.

(iv) \underline{$\widehat{T}_{0}$ is closed in $H\oplus W$}:\newline Suppose
that $\{(x_{n},\Psi x_{n})\}\subseteq\mathcal{D}(\widehat{T}_{0})$ is such
that%
\begin{equation}
(x_{n},\Psi x_{n})\rightarrow(x,a)\text{ in }H\oplus W \label{S1}%
\end{equation}
and%
\begin{equation}
\widehat{T}_{0}(x_{n},\Psi x_{n})\rightarrow(y,b)\text{ in }H\oplus W.
\label{S2}%
\end{equation}
These conditions in $($\ref{S1}$)$ and $($\ref{S2}$)$ are equivalent to%
\begin{equation}
x_{n}\rightarrow x\text{ and }T_{1}x_{n}\rightarrow y\text{ in }H \label{S2.5}%
\end{equation}
and%
\begin{equation}
\Psi x_{n}\rightarrow a\text{ and }B\Psi x_{n}\rightarrow b\text{ in }W.
\label{S3}%
\end{equation}
We need to show that $(x,a)\in\mathcal{D}(\widehat{T}_{0})$ and $\widehat
{T}_{0}(x,a)=(y,b);$ that is to say, we need to prove:%
\begin{equation}
x\in\Delta_{0} \label{S4}%
\end{equation}%
\begin{equation}
T_{1}x=y \label{S5}%
\end{equation}%
\begin{equation}
\Psi x=a \label{S6}%
\end{equation}
and%
\begin{equation}
Ba=b. \label{S7}%
\end{equation}
Since $\{x_{n}\}\subseteq\Delta_{0},$ we see that $T_{1}x_{n}=Jx_{n}$ so, by
Proposition \ref{closure_prop},
\[
x\in\Delta_{0}\text{ and }Jx=T_{1}x=y,
\]
establishing $($\ref{S4}$)$ and $($\ref{S5}$).$ For the remainder of this
proof, write%
\begin{equation}
x=x_{0}+t \label{S8}%
\end{equation}

\noindent and
\begin{equation}
x_{n}=x_{n,0}+t_{n}, \label{S9}%
\end{equation}
where $x_{0},$ $x_{n,0}\in\mathcal{D}(T_{0})$,%
\[
t=\sum_{j=1}^{\dim W}\alpha_{j}t_{j},
\]
and%
\[
t_{n}=\sum_{j=1}^{\dim W}\alpha_{n,j}t_{j}.
\]
From $($\ref{S9}$)$ and the definition of $\Psi,$ we see that
\[
\Psi x_{n}=\Psi(x_{n,0}+t_{n})=\Psi t_{n}=\sum_{j=1}^{\dim W}\alpha_{n,j}%
\xi_{j}%
\]
so that, from $($\ref{S3}$),$
\[
\alpha_{n,j}=(\Psi x_{n},\xi_{j})_{W}\rightarrow\left\langle a,\xi
_{j}\right\rangle _{W}.
\]
It follows that
\begin{equation}
t_{n}\rightarrow\widehat{t}:=\sum_{j=1}^{\dim W}\left\langle a,\xi
_{j}\right\rangle _{W}t_{j} \label{S10}%
\end{equation}
so that $T_{1}t_{n}\rightarrow T_{1}\widehat{t}.$ Notice that
\begin{equation}
\Psi t_{n}\rightarrow\Psi\widehat{t}=\sum_{j=1}^{\dim W}\left\langle a,\xi
_{j}\right\rangle _{W}\xi_{j}=a. \label{S11}%
\end{equation}
From $($\ref{S2.5}$)$, $($\ref{S9}$)$ and $($\ref{S10}$),$ we deduce that
\begin{align*}
x_{n,0}  &  =x_{n}-t_{n}\rightarrow x-\widehat{t}=x_{0}+t-\widehat{t}\text{ in
}H\\
T_{0}x_{n,0}  &  =T_{1}x_{n}-T_{1}t_{n}\rightarrow y-T_{1}\widehat{t}\text{ in
}H.
\end{align*}
Since $T_{0}$ is closed, we see that $x_{0}+t-\widehat{t}\in\mathcal{D}%
(T_{0})$ and, in particular, that $t-\widehat{t}\in\mathcal{D}(T_{0})$. By
definition of the partial GKN set $\{t_{j}\mid1\leq j\leq\dim W\},$ we must
have%
\begin{equation}
t=\widehat{t}. \label{S12}%
\end{equation}
Combining $($\ref{S11}$)$ and $($\ref{S12}$),$ we obtain%
\[
\Psi x_{n}=\Psi t_{n}\rightarrow\Psi\widehat{t}=\Psi t=\Psi(x-x_{0})=\Psi
x=a;
\]
establishing $($\ref{S6}$).$ Finally,%
\[
B\Psi x_{n}\rightarrow B\Psi x=Ba
\]
so, by $($\ref{S3}$),$ $Ba=b$ which proves $($\ref{S7}$)$. This completes the
proof that $\widehat{T}_{0}$ is closed.
\end{proof}

This brings us to the proof of the fundamental relation between the maximal
and minimal operators $\widehat{T}_{1}$ and $\widehat{T}_{0}$.

\begin{theorem}
\label{adjoint_lemma} $(\widehat{T}_{0})^{\ast}=\widehat{T}_{1}.$
\end{theorem}

\begin{proof}
For $(x,a)\ ,(y,b)\in\mathcal{D(}\widehat{T}_{1}),$ a calculation shows that
\begin{equation}
\left\langle \widehat{T}_{1}(x,a),(y,b)\right\rangle _{H\oplus W}-\left\langle
(x,a),\widehat{T}_{1}(y,b)\right\rangle _{H\oplus W}=[x,y]_{H}-\left\langle
\Omega x,b\right\rangle _{W}+\left\langle a,\Omega y\right\rangle _{W}.
\label{symp-form-eq}%
\end{equation}
(see Definition \ref{General Symplectic Form}). Notice that when $y=t_{k}$ and
$b=\xi_{k}$, we obtain
\begin{equation}
\left\langle \Omega x,\xi_{j}\right\rangle _{W}=\sum_{k=1}^{\dim
W}\left\langle [x,t_{k}]_{H}\xi_{k},\xi_{j}\right\rangle _{W}=[x,t_{j}]_{H}
\label{A2}%
\end{equation}
since $\{\xi_{j}\mid j=1,\ldots,\dim W\}$ is an orthonormal basis of $W.$
Suppose now that $(y,b)\in\mathcal{D(}\widehat{T}_{0})$ so $y\in\Delta_{0} $
and $b=\Psi y.$ Then
\[
y=y_{0}+\widetilde{t},
\]
where $y_{0}\in\mathcal{D}(T_{0}),$
\[
\widetilde{t}:=\sum_{j=1}^{\dim W}\alpha_{j}t_{j}%
\]
and%
\[
b=\Psi y=\Psi\widetilde{t}=\sum_{j=1}^{\dim W}\alpha_{j}\xi_{j}.
\]
By $($\ref{T_0 contained in T_1}$),$ $\widehat{T}_{0}(y,\Psi y)=\widehat
{T}_{1}(y,\Psi y)$ so, from $($\ref{symp-form-eq}$),$ we obtain
\begin{align}
&  \left\langle \widehat{T}_{1}(x,a),(y,\Psi y)\right\rangle _{H\oplus
W}-\left\langle (x,a),\widehat{T}_{0}(y,\Psi y)\right\rangle _{H\oplus
W}\nonumber\\
&  =[x,y]_{H}-\left\langle \Omega x,b\right\rangle _{W}+\left\langle a,\Omega
y\right\rangle _{W}\nonumber\\
&  =[x,y_{0}]_{H}+[x,\widetilde{t}]_{H}-\left\langle \Omega x,\Psi
y_{0}\right\rangle _{W}-\left\langle \Omega x,\Psi\widetilde{t}\right\rangle
_{W}+\left\langle a,\Omega y\right\rangle _{W}. \label{A3}%
\end{align}
We now deal with each of the five terms in $($\ref{A3}$).$ First, from Lemma
\ref{minimal domain},
\begin{equation}
\lbrack x,y_{0}]_{H}=0. \label{A4}%
\end{equation}
From $($\ref{A2}$)$, we see that%
\begin{equation}
\lbrack x,\widetilde{t}]_{H}-\left\langle \Omega x,\Psi\widetilde
{t}\right\rangle _{W}=\sum_{j=1}^{\dim W}\alpha_{j}\{[x,t_{j}]_{H}%
-\left\langle \Omega x,\xi_{j}\right\rangle _{W}\}=0. \label{A5}%
\end{equation}
From $($\ref{Psi zero on minimal domain}$),$ $\Psi y_{0}=0$ so
\begin{equation}
\left\langle \Omega x,\Psi y_{0}\right\rangle _{W}=0. \label{A6}%
\end{equation}
Likewise, from $($\ref{Property of Omega}$),$ we see that $\Omega y=0$ so
\begin{equation}
\left\langle a,\Omega y\right\rangle _{W}. \label{A7}%
\end{equation}
Together, $($\ref{A4}$),$ $($\ref{A5}$),$ $($\ref{A6}$)$ and $($\ref{A7}$)$
show that
\begin{equation}
\left\langle \widehat{T}_{1}(x,a),(y,\Psi y)\right\rangle _{H\oplus
W}=\left\langle (x,a),\widehat{T}_{0}(y,\Psi y)\right\rangle _{H\oplus W}%
\quad((x,a)\in\mathcal{D}(\widehat{T}_{1}),\text{ }(y,\Psi y)\in
\mathcal{D(}\widehat{T}_{0})) \label{A8}%
\end{equation}
and, hence, we obtain
\begin{equation}
\widehat{T}_{1}\subseteq(\widehat{T}_{0})^{\ast}. \label{A9}%
\end{equation}
To show $(\widehat{T}_{0})^{\ast}\subseteq\widehat{T}_{1}$, let $(x,a)\in
\mathcal{D}((\widehat{T}_{0})^{\ast})$ and set $(x^{\ast},a^{\ast}%
)=(\widehat{T}_{0})^{\ast}(x,a).$ Then for $(y,\Psi y)\in\mathcal{D(}%
\widehat{T}_{0}),$%
\begin{equation}
\left\langle (x^{\ast},a^{\ast}),(y,\Psi y)\right\rangle _{H\oplus
W}=\left\langle (\widehat{T}_{0})^{\ast}(x,a),(y,\Psi y)\right\rangle
_{H\oplus W}=\left\langle (x,a),\widehat{T}_{0}(y,\Psi y)\right\rangle
_{H\oplus W} \label{A9.5}%
\end{equation}
since $\widehat{T}_{0}$ is closed. Written out, the identity in $($%
\ref{A9.5}$)$ gives%
\begin{equation}
\left\langle x^{\ast},y\right\rangle _{H}+\left\langle a^{\ast},\Psi
y\right\rangle _{W}=\left\langle x,T_{0}y\right\rangle _{H}+\left\langle
a,B\Psi y\right\rangle _{W}. \label{A10}%
\end{equation}
In particular, if $y\in\mathcal{D}(T_{0}),$ then $($\ref{A10}$)$ reduces to
\[
\left\langle x^{\ast},y\right\rangle _{H}=\left\langle x,T_{0}y\right\rangle
_{H}.
\]
Thus $x\in\mathcal{D}(T_{0}^{\ast})=\mathcal{D}(T_{1})$ and \textrm{def}%
\begin{equation}
T_{1}x=x^{\ast}. \label{A11}%
\end{equation}
Substituting $($\ref{A11}$)$ into $($\ref{A10}$)$ and recalling that $B$ is
symmetric in $W$ yields%
\begin{align}
\left\langle a^{\ast},\Psi y\right\rangle _{W}  &  =\left\langle
x,T_{0}y\right\rangle _{H}+\left\langle a,B\Psi y\right\rangle _{W}%
-\left\langle T_{1}x,y\right\rangle _{H}\nonumber\\
&  =-[x,y]_{H}+\left\langle Ba,\Psi y\right\rangle _{W}. \label{A12}%
\end{align}
In particular, let $y=t_{k}$ so $\Psi y=\xi_{k}$. From $($\ref{A2}$),$ we see
that $[x,t_{k}]_{H}=\left\langle \Omega x,\xi_{k}\right\rangle _{W}.$ Hence,
we find that \ref{A12}$)$ becomes%
\begin{equation}
\left\langle a^{\ast},\xi_{k}\right\rangle _{W}=-\left\langle \Omega x,\xi
_{k}\right\rangle _{W}+\left\langle Ba,\xi_{k}\right\rangle _{W}%
\quad(k=1,2,\ldots,\dim W). \label{A13}%
\end{equation}
Since $\{\xi_{k}\mid k=1,\ldots,\dim W\}$ is a basis for $W,$ we can conclude
from $($\ref{A13}$)$ that%
\begin{equation}
a^{\ast}=Ba-\Omega x. \label{A14}%
\end{equation}

\noindent Consequently, from $($\ref{A11}$)$ and $($\ref{A14}$),$ we see that%
\[
(x^{\ast},a^{\ast})=(T_{1}x,Ba-\Omega x)=\widehat{T}_{1}(x,a)
\]
so
\begin{equation}
(\widehat{T}_{0})^{\ast}\subseteq\widehat{T}_{1}. \label{A15}%
\end{equation}
Combining $($\ref{A9}$)$ and $($\ref{A15}$)$, we obtain $(\widehat{T}%
_{0})^{\ast}=\widehat{T}_{1}.$
\end{proof}

Together Theorem \ref{min_op_closed} and Theorem \ref{adjoint_lemma} establish
the following fundamental operator relationship between $\widehat{T}_{0}$ and
$\widehat{T}_{1}.$

\begin{theorem}
\label{max_min_thm}$\widehat{T}_{0}=\overline{\widehat{T}_{0}}\subseteq
(\widehat{T}_{0})^{\ast}=\widehat{T}_{1}.$
\end{theorem}

Consequently we may apply the Stone-von Neumann theory to the minimal operator
$\widehat{T}_{0}.$ Accordingly we define the \textit{positive} and
\textit{negative deficiency spaces} associated with $\widehat{T}_{0}$ in
$H\oplus W$

\begin{definition}
[Deficiency Spaces in the Extended Space $H\oplus W$]\label{def_space_y}
\[
Y_{\pm}:=\{(x,a)\in\mathcal{D(}\widehat{T}_{1})\mid\ \widehat{T}_{1}(x,a)=\pm
i(x,a)\}.
\]

\end{definition}

Remarkably, as we shall see in the next result, the deficiency spaces $Y_{\pm
}$ $\subseteq H\oplus W$and $X_{\pm}\subseteq H$ are \textit{isomorphic}. We
note that, since $B:W\rightarrow W$ is self-adjoint, then $B\pm iI$ is invertible.

\begin{lemma}
\label{def_space_iso} $(x,a)\in Y_{\pm}$ if and only if $x\in X_{\pm}$ and
$a=\left(  B\mp iI\right)  ^{-1}\Omega x.$ Moreover, the deficiency indices of
$\widehat{T}_{0}$ are equal and finite and satisfy \textrm{def}$(\widehat
{T}_{0})=\mathrm{def}(T_{0}).$
\end{lemma}

\begin{proof}
Let $(x,a)\in Y_{\pm}$. Then $T_{1}x=\pm ix$ and $Ba-\Omega x=\pm ia.$
Therefore $x\in X_{\pm}$ and $a=\left(  B\mp iI\right)  ^{-1}\Omega x$.
Conversely if $x\in X_{\pm}$ and $a=\left(  B\mp iI\right)  ^{-1}\Omega x$
then $Ba-\Omega x=\pm ia$ so $\widehat{T}_{1}(x,a)=\pm i(x,a)$. We see that
the mappings $X_{\pm}\rightarrow Y_{\pm}$ given by $x\rightarrow(x,\left(
B\mp iI\right)  ^{-1}\Omega x)$ are vector space isomorphisms. In particular,
$\dim\left(  X_{\pm}\right)  =\dim\left(  Y_{\pm}\right)  .$ This shows that
the deficiency indices of the minimal operator $\widehat{T}_{0}$ are finite
and equal with
\begin{equation}
\dim\left(  Y_{+}\right)  =\dim\left(  Y_{-}\right)  =\text{def}(\widehat
{T}_{0})<\infty. \label{eq_def_indc}%
\end{equation}

\end{proof}

In particular, equation (\ref{eq_def_indc}) guarantees the GKN-EM theorem
applies to $\widehat{T}_{0}$. We now define the (degenerate) symplectic form
in $H\oplus W$ associated with the operators $\widehat{T}_{0}$ and
$\widehat{T}_{1}.$ We remark that, in equation (\ref{symp-form-eq}), we
actually already computed this symplectic form.

\begin{definition}
[General Symplectic Form]\label{general symplectic form}
\begin{equation}
\left[  (x,a),(y,b)\right]  _{H\oplus W}:=[x,y]_{H}-\left\langle \Omega
x,b\right\rangle _{W}+\left\langle a,\Omega y\right\rangle _{W}\qquad
((x,a),(y,b)\in\mathcal{D(}\widehat{T}_{1})), \label{General Symplectic Form}%
\end{equation}
where $[\cdot,\cdot]_{H}$ is the symplectic form defined in $($%
\ref{symplectic_form}$)$ and where the mapping $\Omega$ is defined in
$($\ref{Omega map}$).$
\end{definition}

We are now in position to apply the GKN-EM Theorem (Theorem
\ref{The GKN-EM Theorem}) to the minimal operator $\widehat{T}_{0}$ in
$H\oplus W$ and, as a result, characterize all self-adjoint extensions
(respectively, restrictions) of $\widehat{T}_{0}$ (respectively, the maximal
operator $\widehat{T}_{1}).$

\begin{theorem}
[GKN-EM Theorem in $H\oplus W$]\label{Self-Adjoint Operators in H+W}We have
the following assumptions/definitions\medskip$:$

\begin{enumerate}
\item[(i)] $T_{0}$ and $T_{1}$ are, respectively, the minimal and maximal
operators in $\left(  H,\left\langle \cdot,\cdot\right\rangle _{H}\right)  ,$
called the base $($complex$)$ Hilbert space, with domains $\mathcal{D}(T_{0})
$ and $\mathcal{D}(T_{1});$ $T_{0}$ is a closed, symmetric operator satisfying
$T_{0}\subseteq T_{1}$ with $T_{0}^{\ast}=T_{1}$ and $T_{1}^{\ast}=T_{0};$

\item[(ii)] The deficiency indices of $T_{0}$ are assumed to be equal and
finite and denoted by \textrm{def}$(T_{0});$

\item[(iii)] $[\cdot,\cdot]_{H}$ is the symplectic form given by%
\[
\lbrack x,y]_{H}=\left\langle T_{1}x,y\right\rangle _{H}-\left\langle
x,T_{1}y\right\rangle _{H}\quad(x,y\in\mathcal{D}(T_{1})),
\]
$($see Definition \ref{symplectic_form}$);$

\item[(iv)] $(W,\left\langle \cdot,\cdot\right\rangle _{W}),$ the extension
space, is a finite dimensional complex Hilbert space with \textrm{dim}$W\leq$
\textrm{def}$(T_{0})$ $($Condition \ref{Dimensionality Condition}$)$ and
orthonormal basis $\{\xi_{j}\mid j=1,\ldots,\dim W\};$

\item[(v)] $B:W\rightarrow W$ is a self-adjoint operator;

\item[(vi)] $H\oplus W$, the extended space, is the Hilbert space defined in
$($\ref{H+W space}$)$ with inner product $($\ref{H+W IP}$);$

\item[(vii)] $P=\{t_{j}\mid j=1,\ldots,\dim W\}$ is a partial GKN set $($see
$($\ref{linear independence}$)$ and $($\ref{gkn symmetry}$));$

\item[(viii)] $\Delta_{0}=\mathcal{D}(T_{0})+\mathrm{span}\{t_{j}\mid
j=1,\ldots,\dim W\}$ $($see $($\ref{Delta_0 set}$));$

\item[(ix)] $\Psi:\Delta_{0}\rightarrow W$ is defined to be%
\[
\Psi\left(  x_{0}+\sum_{j=1}^{\dim W}\alpha_{j}t_{j}\right)  =\sum_{j=1}^{\dim
W}\alpha_{j}\xi_{j}\quad(x_{0}\in\mathcal{D}(T_{0}))
\]
$($see $($\ref{Psi zero on minimal domain}$));$

\item[(x)] $\Omega:\mathcal{D}(T_{1})\rightarrow W$ is given by%
\[
\Omega x=\sum_{j=1}^{\dim W}[x,t_{j}]_{H}\xi_{j}%
\]
$($see $($\ref{Omega map}$));$

\item[(xi)] $\widehat{T}_{0}:\mathcal{D}(\widehat{T}_{0})\subseteq H\oplus
W\rightarrow H\oplus W$ is the minimal operator in $H\oplus W$ defined by%
\begin{align*}
\widehat{T}_{0}(x,a)  &  =(T_{1}x,Ba)\\
(x,a)  &  \in\mathcal{D}(\widehat{T}_{0})=\{(x,\Psi x)\mid x\in\Delta_{0}\}
\end{align*}
$($see Definition \ref{min_op}$);$

\item[(xii)] $\widehat{T}_{1}:\mathcal{D}(\widehat{T}_{1})\subseteq H\oplus
W\rightarrow H\oplus W$ is the maximal operator in $H\oplus W$ defined by%
\begin{align*}
\widehat{T}_{1}(x,a)  &  =(T_{1}x,Ba-\Omega x)\\
\mathcal{D}(\widehat{T_{1}})  &  =\{(x,a)\mid x\in\mathcal{D}(T_{1});a\in W\}
\end{align*}
$($see \ref{max_op}$);$

\item[(xiii)] $[\cdot,\cdot]_{H\oplus W}$ is the symplectic form given by%
\[
\left[  (x,a),(y,b)\right]  _{H\oplus W}:=[x,y]_{H}-\left\langle \Omega
x,b\right\rangle _{W}+\left\langle a,\Omega y\right\rangle _{W}\qquad
((x,a),(y,b)\in\mathcal{D(}\widehat{T}_{1}),
\]
$($see \ref{General Symplectic Form}$).\medskip$
\end{enumerate}

Under these definitions and assumptions, we obtain the following results$:$

\begin{enumerate}
\item[(a)] $\widehat{T}_{0}$ \ is a closed, symmetric operator satisfying
$\widehat{T}_{0}\subseteq\widehat{T}_{1}$ with $(\widehat{T}_{0})^{\ast
}=\widehat{T}_{1}$ and $(\widehat{T}_{1})^{\ast}=\widehat{T}_{0}$ $($see
Theorems \ref{min_op_closed} and \ref{adjoint_lemma}$);$

\item[(b)] The deficiency indices of $\widehat{T}_{0}$ are equal and finite
and \textrm{def}$(\widehat{T}_{0})=\mathrm{def}(T_{0})$ $($see Lemma
\ref{def_space_iso}$);$

\item[(c)] Suppose $\widehat{T}$ is a self-adjoint extension of $\widehat
{T}_{0}$ $($equivalently, $\widehat{T}$ is a self-adjoint restriction of
$\widehat{T}_{1})$ satisfying $\widehat{T}_{0}\subseteq\widehat{T}$
$\mathcal{\subseteq}$ $\widehat{T}_{1}.$ Then there exists a GKN set
$\{(x_{j},a_{j})\mid j=1,\ldots,\mathrm{def}(T_{0})\}\subseteq\mathcal{D}%
(\widehat{T}_{1})$ $($see Remark \ref{Remark 2}$)$ satisfying the two conditions

\begin{enumerate}
\item[($\alpha$)] $\{(x_{j},a_{j})\mid j=1,\ldots,\mathrm{def}(T_{0})\}$ is
linearly independent modulo $\mathcal{D}(\widehat{T}_{0}),$

\item[($\beta$)] $[(x_{j},a_{j}),(x_{k},a_{k})]_{H\oplus W}=0$ for
$j,k=1,\ldots,\mathrm{def}(T_{0})$
\end{enumerate}

\noindent such that%
\begin{align}
\widehat{T}(x,a)  &  =(T_{1}x,Ba-\Omega x)\label{F1}\\
\mathcal{D}(\widehat{T})  &  =\{(x,a)\in\mathcal{D}(\widehat{T}_{1}%
)\mid\lbrack(x,a),(x_{j},a_{j})]_{H\oplus W}=0\text{ }(j=1,\ldots
,\mathrm{def}(T_{0}))\}. \label{F2}%
\end{align}

\item[(d)] If $\widehat{T}$ is defined by $($\ref{F1}$)$ and $($\ref{F2}$)$
where $\{(x_{j},a_{j})\mid j=1,\ldots,\mathrm{def}(T_{0})\}\subseteq
\mathcal{D}(\widehat{T}_{1})$ is a GKN set satisfying conditions $(\alpha)$
and $(\beta),$ then $\widehat{T}$ is a self-adjoint extension of $\widehat
{T}_{0}$ $($equivalently, $\widehat{T}$ is a self-adjoint restriction of
$\widehat{T}_{1})$ in $H\oplus W.$
\end{enumerate}
\end{theorem}

\section{Examples\label{Examples}}

\subsection{Example 1: The Legendre Type Self-Adjoint
Operator\label{Section Four}}

Throughout this example, we let $H=L^{2}(-1,1)$ (with its usual inner product)
and $W=\mathbb{C}^{2},$ endowed with the weighted Euclidean inner product%
\begin{equation}
\left\langle (a_{1},b_{1}),(a_{2},b_{2})\right\rangle _{W}=\dfrac
{a_{1}\overline{a}_{2}+b_{1}\overline{b}_{2}}{A}\quad((a_{1},b_{1}%
),(a_{2},b_{2})\in W); \label{Weighted C^2 IP}%
\end{equation}
here $A$ is a fixed, positive constant. Let $\{\xi_{1},\xi_{2}\}$ be the
orthonormal basis in $W$ given by%
\begin{equation}
\xi_{1}=(\sqrt{A},0)\text{ and }\xi_{2}=(0,\sqrt{A})
\label{Orthonormal Basis for W}%
\end{equation}
and, for this example, suppose $B:W\rightarrow W$ is the zero self-adjoint operator.

In \cite{EKL-QM, EL-DIE, AMKrall-1981}, the authors discuss the spectral
analysis of the Legendre type differential expression defined earlier in
(\ref{legendre type equation}); that is,
\begin{equation}
\ell_{LT}[y](u):=\left(  (1-u^{2})^{2}y^{\prime\prime}(u)\right)
^{\prime\prime}-\left(  (8+4A(1-u^{2}))y^{\prime}(u)\right)  ^{\prime}
\label{Legendre type DE}%
\end{equation}
where $A$ is the same constant appearing in (\ref{Weighted C^2 IP}). This
differential expression was first discovered by H.\ L. Krall
\cite{HLKrall-1938, HLKrall-1940}. When
\begin{equation}
\lambda_{n}=n(n+1)(n^{2}+n+4A-2)\quad(n\in\mathbb{N}_{0}),
\label{Eigenvalues of Legendre type}%
\end{equation}
the equation $\ell_{LT}[y]=\lambda_{n}y$ has a polynomial solution
$y=P_{n,A}(u)$ of degree $n;$ that is,
\begin{equation}
\ell_{LT}[P_{n,A}]=\lambda_{n}P_{n,A}\quad(n\in\mathbb{N}_{0}).
\label{LT eigenvalue equation}%
\end{equation}
The sequence $\{P_{n,A}\}_{n=0}^{\infty}$ is called the \textit{Legendre type
polynomials}; they form a complete orthogonal sequence in the Hilbert space
$L_{\mu}^{2}[-1,1],$ where%
\[
L_{\mu}^{2}[-1,1]=\{f:[-1,1]\rightarrow\mathbb{C}\mid f\text{ is Lebesgue
measurable with }\int_{[-1,1]}\left\vert f\right\vert ^{2}d\mu<\infty\},
\]
with inner product%
\[
\left\langle f,g\right\rangle _{\mu}=\dfrac{f(-1)\overline{g}(-1)}{A}%
+\int_{-1}^{1}f(u)\overline{g}(u)du+\dfrac{f(1)\overline{g}(1)}{A},
\]
and where $d\mu$ is the Lebesgue-Stieltjes measure given by
\[
d\mu=dx+\dfrac{1}{A}\delta(x+1)+\dfrac{1}{A}\delta(x-1).
\]
When $y=P_{n,A},$ we see from (\ref{Legendre type DE}) and
(\ref{LT eigenvalue equation}) that%
\begin{equation}
\mp8AP_{n,A}^{\prime}(\pm1)=\lambda_{n}P_{n,A}(\pm1).
\label{Legendre Type Polynomial identity}%
\end{equation}
Various properties of the Legendre type polynomials can be found in
\cite{AMKrall-1981}.

Because the measure $\mu$ has jumps at $u=\pm1,$ the classic GKN theory is not
immediately applicable in finding a self-adjoint operator representation $T$
of $\ell_{LT}[\cdot]$ in $L_{\mu}^{2}[-1,1].$ In order to construct $T,$
Everitt and Littlejohn \cite{EKL-QM, EL-DIE} first studied properties of
functions in the maximal domain
\[
\mathcal{D}(T_{1})=\{x:(-1,1)\rightarrow\mathbb{C}\mid x,x^{\prime}%
,x^{\prime\prime},x^{\prime\prime\prime}\in AC_{\text{\textrm{loc}}%
}(-1,1);x,\ell_{LT}[x]\in L^{2}(-1,1)\},
\]
where $T_{1}$ is the maximal operator, generated by $\ell_{LT}[\cdot],$ in the
Hilbert space $L^{2}(-1,1).$ They establish the remarkable smoothness
property
\begin{equation}
x\in\mathcal{D}(T_{1})\Longrightarrow x^{\prime\prime}\in L^{2}(-1,1)
\label{L2 smoothness}%
\end{equation}
and hence, upon making the natural identifications
\begin{equation}
x(\pm1)=\lim_{u\rightarrow\pm1}x(u)\qquad x^{\prime}(\pm1)=\lim_{u\rightarrow
\pm1}x^{\prime}(u), \label{endpoint identification}%
\end{equation}
we can say that
\[
x\in\mathcal{D}(T_{1})\Longrightarrow x,x^{\prime}\in AC[-1,1].
\]
Moreover, they prove that the associated sesquilinear form has the simple
formulation
\begin{equation}
\lbrack x,y]_{H}=8(x(1)\overline{y}^{\prime}(1)-x^{\prime}(1)\overline
{y}(1)+x^{\prime}(-1)\overline{y}(-1)-x(-1)\overline{y}^{\prime}%
(-1))\qquad(x,y\in\mathcal{D}(T_{1})). \label{sesquilinear form 4}%
\end{equation}
Considering this last formula and Lemma \ref{minimal domain}, it is apparent
that the minimal domain associated with $\ell_{LT}[\cdot]$ is explicitly given
by
\begin{equation}
\mathcal{D}(T_{0})=\left\{  x\in\mathcal{D}(T_{1})\mid x(\pm1)=x^{\prime}%
(\pm1)=0\right\}  . \label{delta min 4}%
\end{equation}
The deficiency index of the minimal operator $T_{0},$ generated by $\ell
_{LT}[\cdot],$ in $L^{2}(-1,1)$ is \textrm{def}$(T_{0})=2.$ This follows since
each endpoint $u=\pm1$ is in the limit-3 case which can be shown by a
Frobenius analysis. We emphasize that we are not seeking to find self-adjoint
extensions of $T_{0}$ in $L^{2}(-1,1)$ but instead we want to find a
self-adjoint representation of $\ell_{LT}[\cdot]$ in $L_{\mu}^{2}[-1,1]$ which
produces the Legendre type polynomials $\{P_{n,A}\}_{n=0}^{\infty}$ as
eigenfunctions. By analyzing functions in $\mathcal{D}(T_{1}),$ Everitt and
Littlejohn show that the operator $T:\mathcal{D}(T)\subseteq$ $L_{\mu}%
^{2}[-1,1]\rightarrow L_{\mu}^{2}[-1,1]$ defined by%
\begin{align}
Tx(u)  &  =\left\{
\begin{array}
[c]{ll}%
-8Ax^{\prime}(-1) & u=-1\\
\ell_{LT}[x](u) & -1<u<1\\
8Ax^{\prime}(1) & u=1
\end{array}
\right. \label{Self-Adjoint operator T in Jump Space}\\
x  &  \in\mathcal{D}(T):=\mathcal{D}(T_{1})\nonumber
\end{align}
is self-adjoint, has the Legendre type polynomials $\{P_{n,A}\}_{n=0}^{\infty
}$ as eigenfunctions, and has discrete spectrum
\[
\sigma(T)=\sigma_{p}(T)=\{\lambda_{n}\mid n\in\mathbb{N}_{0}\},
\]
where each $\lambda_{n}$ is given in (\ref{Eigenvalues of Legendre type}). It
is surprising that the maximal domain $\mathcal{D}(T_{1})$ is the domain of a
self-adjoint operator in $L_{\mu}^{2}[-1,1]$. By the GKN Theorem,
$\mathcal{D}(T_{1})$ cannot be the domain of a self-adjoint extension of
$T_{0}$ in $L^{2}(-1,1).$

We now show, using the results developed in this paper, how to construct the
self-adjoint operator $T$ given in
(\ref{Self-Adjoint operator T in Jump Space}) in the direct sum space $H\oplus
W.$ Indeed, below, we construct a self-adjoint operator $\widehat{T}$ that is,
essentially, the operator $T$ defined in
(\ref{Self-Adjoint operator T in Jump Space}). With this alternative approach,
we will see how \textit{continuity} is a GKN-EM boundary condition that
produces the Legendre type self-adjoint operator $\widehat{T}.$

The first step in our analysis is to observe that the space $L_{\mu}%
^{2}[-1,1]$ is isometrically isomorphic to the direct sum
\[
H\oplus W=\left\{  (x,(a,b))\mid x\in H;(a,b)\in W\right\}  .
\]

Next define $t_{j}\in\mathcal{D}(T_{1})$ $(j=1,2)$ by%
\[
t_{1}(u)=\left\{
\begin{array}
[c]{ll}%
\sqrt{A} & u\text{ near }-1\\
0 & u\text{ near }1
\end{array}
\right.  \text{ }t_{2}(u)=\left\{
\begin{array}
[c]{ll}%
0 & u\text{ near }-1\\
\sqrt{A} & u\text{ near }1;
\end{array}
\right.
\]
we remark that such functions in $\mathcal{D}(T_{1})$ exist by Naimark's
Patching Lemma \cite[Lemma 2, Section 17.3]{Naimark}. It is straightforward to
see, using (\ref{sesquilinear form 4}) and (\ref{delta min 4}), that
$\{t_{1},t_{2}\}$ is a GKN set for $T_{0}.$ Consequently, we see that%
\begin{equation}
\Delta_{0}=\{x_{0}+c_{1}t_{1}+c_{2}t_{2}\mid x_{0}\in\mathcal{D}(T_{0}%
);c_{1},c_{2}\in\mathbb{C}\}, \label{Delta_0}%
\end{equation}
where $\Delta_{0}$ is defined in (\ref{Delta_0 set}). Moreover,%
\begin{equation}
\Psi(x_{0}+c_{1}t_{1}+c_{2}t_{2})=c_{1}\xi_{1}+c_{2}\xi_{2}=\langle c_{1}%
\sqrt{A},c_{2}\sqrt{A}\rangle, \label{Psi}%
\end{equation}
where $\{\xi_{1},\xi_{2}\}$ is defined in (\ref{Orthonormal Basis for W}) and
where $\Psi:\Delta_{0}\rightarrow W$ is the map defined in
(\ref{Psi zero on minimal domain}). Using (\ref{sesquilinear form 4}),
calculations show that%
\[
\lbrack x,t_{1}]_{H}\text{ }=8\sqrt{A}x^{\prime}(-1),\text{ }[x,t_{2}%
]_{H}=-8\sqrt{A}x^{\prime}(1)\quad(x\in\mathcal{D}(T_{1})).
\]
It follows that%
\begin{equation}
\Omega x=8\sqrt{A}x^{\prime}(-1)\xi_{1}-8\sqrt{A}x^{\prime}(1)\xi
_{2}=(8Ax^{\prime}(-1),-8Ax^{\prime}(1))\quad(x\in\mathcal{D}(T_{1})),
\label{Omega f}%
\end{equation}
where $\Omega:\mathcal{D}(T_{1})\rightarrow W$ is the mapping defined in
(\ref{Omega map}).

The minimal operator $\widehat{T}_{0}:\mathcal{D}(\widehat{T}_{0})\subseteq$
$H\oplus W\rightarrow H\oplus W$, in this example, is given by%
\begin{align}
\widehat{T}_{0}(x,\Psi x)  &  =(T_{1}x,B\Psi x)=(\ell_{LT}%
[x],(0,0))\label{T_0 for Legendre type}\\
\mathcal{D}(\widehat{T}_{0})  &  =\{(x,\Psi x)\mid x\in\Delta_{0}\}.
\label{Domain of T_0 for Legendre type}%
\end{align}
From the theory we established in Section \ref{Section Three}, $\widehat
{T}_{0}$ is a closed, symmetric operator in $H\oplus W$ with \textrm{def}%
$\widehat{T}_{0}=2.$

Using (\ref{Omega f}), we see that the associated maximal operator
$\widehat{T}_{1}:\mathcal{D}(\widehat{T}_{1})\subseteq$ $H\oplus W\rightarrow
H\oplus W$ is given explicitly by%
\begin{align}
\widehat{T}_{1}(x,(a,b))  &  =(\ell_{LT}[x],(-8Ax^{\prime}(-1),8Ax^{\prime
}(1)))\label{T_1 for Legendre type}\\
\mathcal{D}(\widehat{T}_{1})  &  =\{(x,(a,b))\mid x\in\mathcal{D}%
(T_{1});\text{ }a,b\in\mathbb{C}\}. \label{Domain of T_1 for Legendre type}%
\end{align}
From (\ref{Weighted C^2 IP}), (\ref{sesquilinear form 4}) and (\ref{Omega f}),
a calculation shows that the symplectic form $[\cdot,\cdot]_{H\oplus W}, $
defined in (\ref{General Symplectic Form}), is given by%
\begin{align}
&  [(x,(a_{1},b_{1})),(y,(a_{2},b_{2}))]_{H\oplus W}\nonumber\\
&  =[x,y]_{H}-\left\langle \Omega x,(a_{2},b_{2})\right\rangle _{W}%
+\left\langle (a_{1},b_{1}),\Omega y\right\rangle _{W}\nonumber\\
&  =8(x(1)\overline{y}^{\prime}(1)-x^{\prime}(1)\overline{y}(1)+x^{\prime
}(-1)\overline{y}(-1)-x(-1)\overline{y}^{\prime}%
(-1))\label{Legendre type symplectic form}\\
&  -8x^{\prime}(-1)\overline{a}_{2}+8x^{\prime}(1)\overline{b}_{2}%
+8Ay^{\prime}(-1)a_{1}-8Ay^{\prime}(1)b_{1}\nonumber
\end{align}
for $(x,(a_{1},b_{1})),$ $(y,(a_{2},b_{2}))\in\mathcal{D}(\widehat{T}_{1}).$
Define $x_{j}\in\mathcal{D}(T_{1})$, for $j=1,2,$ by%
\[
x_{1}(u)=\left\{
\begin{array}
[c]{ll}%
0 & u\text{ near }-1\\
\sqrt{A}(u-1) & u\text{ near }1
\end{array}
\right.  ,\text{ }x_{2}(u)=\left\{
\begin{array}
[c]{ll}%
\sqrt{A}(u+1) & u\text{ near }-1\\
0 & u\text{ near }1.
\end{array}
\right.
\]
From (\ref{sesquilinear form 4}) and (\ref{delta min 4}), we see that
$\{x_{1},x_{2}\}$ is a GKN set for $T_{0}.$ Calculations also show that%
\begin{equation}
\lbrack x,x_{1}]_{H}=8\sqrt{A}x(1)\text{ and }[x,x_{2}]_{H}=-8\sqrt
{A}x(-1)\quad(x\in\mathcal{D}(T_{1})). \label{Calculations with x_1 and x_2}%
\end{equation}
In addition, we see that%
\begin{equation}
\Omega x_{1}=(0,-8A\sqrt{A})\text{ and }\Omega x_{2}=(8A\sqrt{A},0).
\label{Calculations with Omega x_1 and x_2}%
\end{equation}
We now claim that

\begin{proposition}
\label{Proposition 1} $\{(x_{1},(0,0)),(x_{2},(0,0))\}$ is a GKN set for
$\widehat{T}_{0}.$

\begin{proof}
Suppose that%
\[
c_{1}(x_{1},(0,0))+c_{2}(x_{2},(0,0))=(c_{1}x_{1}+c_{2}x_{2},(0,0))\in
\mathcal{D(}\widehat{T}_{0}).
\]
By definition of $\mathcal{D}(\widehat{T}_{0}),$ we see that $\Psi(c_{1}%
x_{1}+c_{2}x_{2})=(0,0).$ This implies that $c_{1}x_{1}+c_{2}x_{2}%
\in\mathcal{D}(T_{0}).$ However since $\{x_{1},x_{2}\}$ is a GKN set for
$T_{0}$, we must have $c_{1}=c_{2}=0$. Hence $\{(x_{1},(0,0)),\langle
x_{2},(0,0))\} $ is linearly independent modulo $\mathcal{D}(\widehat{T}%
_{0}).$ Next, using $($\ref{Legendre type symplectic form}$),$ we see that%
\[
[(x_{1},(0,0)),(x_{2},(0,0))]_{H\oplus W}=[x_{1},x_{2}]_{H}=0.
\]
Calculations also show that%
\[
\lbrack(x_{j},(0,0)),(x_{j},(0,0))]_{H\oplus W}=0\quad(j=1,2).
\]
This completes the proof of the Proposition.
\end{proof}
\end{proposition}

We now find the appropriate self-adjoint operator $\widehat{T}$ in $H\oplus W
$ having the Legendre type polynomial vectors $\{(P_{n,A},(P_{n,A}%
(-1),P_{n,A}(1)))\}_{n=0}^{\infty}$ as eigenfunctions. Indeed, using Theorem
\ref{Self-Adjoint Operators in H+W} part (d), the operator $\widehat
{T}:\mathcal{D}(\widehat{T})\subseteq$ $H\oplus W\rightarrow H\oplus W$ ,
defined by
\begin{align}
\widehat{T}(x,(a,b))  &  =(\ell_{LT}[x],(-8Ax^{\prime}(-1),8Ax^{\prime
}(1)))\label{T for LT in Jump Space}\\
\mathcal{D}(\widehat{T})  &  =\{(x,(a,b))\in\mathcal{D}(\widehat{T}_{1}%
)\mid\lbrack(x,(a,b)),(x_{j},(0,0))]_{H\oplus W}=0\text{ }(j=1,2)\}
\label{Domain of T in Jump Space}%
\end{align}
is self-adjoint.

We now investigate each of the two boundary conditions in
(\ref{Domain of T in Jump Space}). From (\ref{Legendre type symplectic form}),
(\ref{Calculations with x_1 and x_2}) and
(\ref{Calculations with Omega x_1 and x_2}), a calculation shows that%
\begin{align*}
0  &  =[(x,(a,b)),(x_{1},(0,0))]_{H\oplus W}\\
&  =[x,x_{1}]_{H}-\left\langle \Omega x,(0,0)\right\rangle _{W}+\left\langle
(a,b),\Omega x_{1}\right\rangle _{W}\\
&  =[x,x_{1}]_{H}+\langle(a,b),(0,-8A\sqrt{A})\rangle_{W}\\
&  =8\sqrt{A}x(1)-\dfrac{8A\sqrt{A}b}{A}=8\sqrt{A}x(1)-8\sqrt{A}b,
\end{align*}
implying
\begin{equation}
b=x(1). \label{b=x(1)}%
\end{equation}

\noindent A similar calculation, using (\ref{Calculations with x_1 and x_2})
and (\ref{Calculations with Omega x_1 and x_2}), yields%
\[
0=[(x,(a,b)),(x_{2},(0,0))]_{H\oplus W}=-8\sqrt{A}x(-1)+8\sqrt{A}a
\]
which establishes%
\begin{equation}
a=x(-1). \label{a=x(-1)}%
\end{equation}
Hence the domain of $\widehat{T},$ given in (\ref{Domain of T in Jump Space}),
simplifies to
\begin{equation}
\mathcal{D}(\widehat{T})=\{(x,(x(-1),x(1)))\mid x\in\mathcal{D}(T_{1})\}.
\label{Domain of T in Jump Space 2}%
\end{equation}
Notice that this domain (\ref{Domain of T in Jump Space 2}) extends the
continuity of each $x\in\mathcal{D}(T_{1})$ from $(-1,1)$ to the closure
$[-1,1].$ It is remarkable that, in this sense, \textit{continuity} is a
GKN-EM boundary condition. Furthermore, from (\ref{LT eigenvalue equation}),
(\ref{Legendre Type Polynomial identity}) and (\ref{T for LT in Jump Space}),
notice that%
\begin{align*}
\widehat{T}(P_{n,A},(P_{n,A}(-1),P_{n,A}(1)))  &  =(\ell_{LT}[P_{n,A}%
],(-8AP_{n,A}^{\prime}(-1),8AP_{n,A}^{\prime}(1)))\\
&  =(\lambda_{n}P_{n,A},(\lambda_{n}P_{n,A}(-1),\lambda_{n}P_{n,A}(1)))\\
&  =\lambda_{n}(P_{n,A},(P_{n,A}(-1),P_{n,A}(1))).
\end{align*}
Moreover, $\widehat{T}$ is the same operator as $T$ defined in
(\ref{Self-Adjoint operator T in Jump Space}).

\begin{remark}
If $B:W\rightarrow W$ is an arbitrary self-adjoint operator in $W,$ the
operator $\widehat{S}:\mathcal{D}(\widehat{S})\subseteq H\oplus W\rightarrow
H\oplus W,$ defined by%
\begin{align*}
\widehat{S}(x,(a,b))  &  =(\ell_{LT}[x],B(a,b)+(-8Ax^{\prime}(-1),8Ax^{\prime
}(1)))\\
\mathcal{D}(\widehat{S})  &  =\{(x,(x(-1),x(1)))\mid x\in\mathcal{D}(T_{1})\},
\end{align*}
is self-adjoint in $H\oplus W$. However, it is the case that the Legendre type
polynomial vectors $\{(P_{n,A},(P_{n,A}(-1),P_{n,A}(1)))\}_{n=0}^{\infty}$ are
eigenfunctions of $\widehat{S}$ if and only if $B=0.$
\end{remark}

\subsection{Example 2: A Simple First-Order Differential Operator}

Let $H=L^{2}[0,1]$ be endowed with the standard $L^{2}$ inner product%
\[
\left\langle x,y\right\rangle _{H}=\int_{-1}^{1}x(u)\overline{y}(u)du
\]
and let $W=\mathbb{C}$ have the usual Euclidean inner product%
\[
\left\langle a,b\right\rangle _{W}:=a\overline{b}\quad(a,b\in W).
\]
In this example, we show how to construct a self-adjoint operator in $H\oplus
W$ generated by the first-order Lagrangian symmetric differential expression
\[
\ell\lbrack x](u)=ix^{\prime}(u).
\]
Our construction can be modified to find numerous other self-adjoint operators
in $H\oplus W$ generated by $\ell\lbrack\cdot].$

The maximal and minimal domains in $H$ associated with $\ell\lbrack\cdot]$ are
respectively given by%
\begin{align*}
T_{1}x  &  =ix^{\prime}\\
\mathcal{D}(T_{1})  &  =\{x:[0,1]\rightarrow\mathbb{C}\mid x\in AC[0,1];\text{
}x^{\prime}\in H\},
\end{align*}
and%
\begin{align}
T_{0}x  &  =ix^{\prime}\nonumber\\
\mathcal{D}(T_{0})  &  =\{x\in\mathcal{D}(T_{1})\mid x(0)=x(1)=0\};
\label{Ex 2 - minimal domain}%
\end{align}
see \cite[Chapter 13, Example 13.4]{Rudin}. The symplectic form $[\cdot
,\cdot]_{H}$ associated with $T_{1}$is given by%
\[
\lbrack x,y]_{H}=ix(1)\overline{y}(1)-ix(0)\overline{y}(0).
\]
An elementary calculation shows that the deficiency indices of $T_{0}$ are
equal with $\mathrm{def}(T_{0})=1.$

Choose $\{\xi_{1}=1\}$ as the orthonormal basis for $W$. All self-adjoint
operators $B:W\rightarrow W$ have the form $Ba=\alpha a$ for some real number
$\alpha;$ we fix one such an operator. Define $t_{1}\in\mathcal{D}(T_{1})$ by%
\[
t_{1}(u)\equiv1\quad(0\leq u\leq1)
\]
and note that
\begin{equation}
\lbrack x,t_{1}]_{H}=i(x(1)-x(0)). \label{Ex2 - H symplectic form}%
\end{equation}
It is clear, from (\ref{Ex 2 - minimal domain}), that $t_{1}$ is not the
minimal domain; moreover$,$ from (\ref{Ex2 - H symplectic form}), we see that
\[
\lbrack t_{1},t_{1}]_{H}=0;
\]
this shows that $\{t_{1}\}$ is a GKN set for $T_{0}$ in $H.$ Moreover, with
this GKN set, the reader can readily verify, using Theorem
\ref{The GKN Theorem}, that the operator $T:\mathcal{D}(T)\subset H\rightarrow
H$ defined by%
\begin{align*}
Tx  &  =ix^{\prime}\\
\mathcal{D}(T)  &  =\{x\in\mathcal{D}(T_{1})\mid x(0)=x(1)\}
\end{align*}
is self-adjoint in $H;$ see \cite[Chapter 13, Example 13.4]{Rudin} for an
interesting direct proof of the self-adjointness of $T.$

The operators $\Omega:\mathcal{D}(T_{1})\rightarrow W$ and $\Psi:\Delta
_{0}\rightarrow W$ from Section \ref{Section Three} are given by%
\[
\Omega x=[x,t_{1}]_{H}\xi_{1}=i(x(1)-x(0))
\]
and
\[
\Psi(t_{0}+at_{1})=a\quad(t_{0}\in\mathcal{D}(T_{0})).
\]
The maximal operator $\widehat{T}_{1}:\mathcal{D}(\widehat{T}_{1})\subset
H\oplus W\rightarrow$ $H\oplus W$ and the minimal operator $\widehat{T}%
_{0}:\mathcal{D}(\widehat{T}_{0})\subset H\oplus W\rightarrow$ $H\oplus W$ can
now both be defined. Indeed,%
\begin{align*}
\widehat{T}_{1}(x,a)  &  =(ix^{\prime},Ba-\Omega x)=(ix^{\prime},\alpha
a-i(x(1)-x(0)))\\
\mathcal{D}(\widehat{T}_{1})  &  =\{(x,a)\mid x\in\mathcal{D}(T_{1}),a\in W\}
\end{align*}
and%
\begin{align*}
\widehat{T}_{0}(x,\Psi x)  &  =(ix^{\prime},B\Psi x)=(ix^{\prime},\alpha\Psi
x)\\
\mathcal{D}(\widehat{T}_{0})  &  =\{(x,\Psi x)\mid x\in\Delta_{0}\}.
\end{align*}

\noindent The symplectic form $[\cdot,\cdot]_{H\oplus W}$ associated with
$\widehat{T}_{1}$ is given by%
\begin{align}
\lbrack(x,a),(y,b)]_{H\oplus W}  &  =[x,y]_{H}-\left\langle \Omega
x,b\right\rangle _{W}+\left\langle a,\Omega y\right\rangle _{W}%
\label{Ex2 - H+W symplectic form}\\
&  =i(x(1)\overline{y}(1)-x(0)\overline{y}(0))-i(x(1)-x(0))\overline
{b}-ai(\overline{y}(1)-\overline{y}(0))\nonumber
\end{align}
for $(x,a),$ $(y,b)\in\mathcal{D}(\widehat{T}_{1}).$

Define $x_{1}\in\mathcal{D}(T_{1})$ by%
\[
x_{1}(u)=\left\{
\begin{array}
[c]{ll}%
1 & u\text{ near }1\\
0 & u\text{ near }0;
\end{array}
\right.
\]
from (\ref{Ex2 - H+W symplectic form}), we see that
\begin{equation}
\lbrack(x,a),(x_{1},1/2)]_{H\oplus W}=ix(1)-i(x(1)-x(0))/2-ai.
\label{Ex2 - Boundary Value}%
\end{equation}
We now show that $\{(x_{1},1/2)\}$ is a GKN set for $\widehat{T_{0}}.$ From
(\ref{Ex2 - Boundary Value}), note that
\begin{equation}
\lbrack(x_{1},1/2),(x_{1},1/2)]_{H\oplus W}=0;
\label{Ex2 - symmetry condition}%
\end{equation}
We claim that $(x_{1},1/2)\notin\mathcal{D}(\widehat{T}_{0});$ indeed,
otherwise, we have%
\begin{equation}
x_{1}(u)=t_{0}(u)+\dfrac{1}{2}t_{1}(u)\quad(u\in\lbrack0,1]) \label{Ex2 - rep}%
\end{equation}
for some $t_{0}\in\mathcal{D}(T_{0}).$ However, choosing $u=0$ or $u=1$ shows
that (\ref{Ex2 - rep}) is not possible. It now follows that $\left\{
(x_{1},1/2)\right\}  $ is a GKN set for $\widehat{T}_{0}.$

From (\ref{Ex2 - Boundary Value}), we see that
\begin{equation}
\lbrack(x,a),(x_{1},1/2)]_{H\oplus W}=0\text{ if and only if }a=\dfrac
{x(0)+x(1)}{2}. \label{Ex2 - BC}%
\end{equation}
It now follows, from Theorem \ref{The GKN-EM Theorem} and (\ref{Ex2 - BC}),
that the operator $\widehat{T}:\mathcal{D}(\widehat{T})\subset H\oplus
W\rightarrow H\oplus W$ defined by%
\begin{align}
\widehat{T}(x,(x(0)+x(1))/2)  &  =(ix^{\prime},\alpha
(x(0)+x(1))/2-i(x(1)-x(0)))\label{Ex2 - T}\\
\mathcal{D}(\widehat{T})  &  =\{(x,(x(0)+x(1))/2)\mid x\in\mathcal{D}%
(T_{1})\}\nonumber
\end{align}
is self-adjoint.

\subsection{Example 3: Variations on the Fourier Self-Adjoint
Operator\label{Section Five}}

For this example, we consider the well known Fourier differential expression%
\begin{equation}
\ell_{F}[y](u)=-y^{\prime\prime}(u)\quad(u\in\lbrack a,b]) \label{Fourier DE}%
\end{equation}
where $[a,b]$ is a compact interval. Here, the Hilbert space is $H=L^{2}[a,b]
$ and, in the sub-examples below, we will consider $W$ to be either
$\mathbb{C}$ or $\mathbb{C}^{2}$ with a weighted Euclidean inner product.

The maximal operator $T_{1}:\mathcal{D}(T_{1})\subset H\rightarrow H$ is
defined by%
\begin{align*}
T_{1}x  &  =\ell_{F}[x]\\
\mathcal{D}(T_{1})  &  =\{x:[a,b]\rightarrow\mathbb{C\mid}x,x^{\prime}\in
AC[a,b];x^{\prime\prime}\in L^{2}[a,b]\}
\end{align*}
while the minimal operator $T_{0}:\mathcal{D}(T_{0})\subset H\rightarrow H$ is
given by%
\begin{align*}
T_{0}x  &  =\ell_{F}[x]\\
\mathcal{D}(T_{0})  &  =\{x\in\mathcal{D}(T_{1})\mathbb{\mid}x(a)=x^{\prime
}(a)=x(b)=x^{\prime}(b)=0\}.
\end{align*}
The symplectic form $[\cdot,\cdot]_{H}$ associated with $T_{1}$ is given by%
\begin{equation}
\lbrack x,y]_{H}=x(b)\overline{y}^{\prime}(b)-x^{\prime}(b)\overline
{y}(b)+x^{\prime}(a)\overline{y}(a)-x(a)\overline{y}^{\prime}(a)\quad
(x,y\in\mathcal{D}(T_{1})). \label{Example 3 Symplectic Form H}%
\end{equation}
Because $\ell_{F}[\cdot]$ is regular, the deficiency index of $T_{0}$ is
\textrm{def}$(T_{0})=2.$ Consequently, by the GKN Theorem, every self-adjoint
extension of $T_{0}$ in $H$ will be a certain restriction of the maximal
operator defined by two appropriate boundary conditions. One such self-adjoint
operator\ is the classical \textit{Fourier trigonometric self-adjoint
operator} in $H,$ generated by $\ell_{F}[f],$ with domain
\[
\left\{  x\in\mathcal{D}(T_{1})\mid x^{\prime}(a)=x^{\prime}%
(b),x(a)=x(b)\right\}  .
\]
We list several examples of self-adjoint operators, generated by $\ell
_{F}[\cdot]$ in $H\oplus W$.

\subsubsection{One Dimensional Extension Spaces}

Consider the one dimensional extension space $W,$ with basis $\{\xi_{1}=1\},$
given by
\begin{gather*}
W=\mathbb{C}\\
\left\langle z_{1},z_{2}\right\rangle _{W}=z_{1}\overline{z_{2}}\quad
(z_{1},z_{2}\in W).
\end{gather*}
Every self-adjoint operator $B:W\rightarrow W$ has the form $Bz=\alpha z$ for
some $\alpha\in\mathbb{R};$ for this example, we fix such a $B.$ Observe the
Hilbert space $H\oplus W$ is suggested in a natural way by the inner product
\[
\int_{a}^{b}x(u)\overline{y}(u)dx+x(b)\overline{y}(b).
\]
With this particular inner product in mind, $H\oplus W$ is isomorphic to the
Lebesgue-Stieltjes integration space generated by the discontinuous
Lebesgue-Stieltjes measure
\[
d\mu=du+\delta(u-b).
\]

\noindent\textbf{Example 3.1 }With the partial GKN set $\{t_{1}\}$ for $T_{0}$
in $H$ given by%
\[
t_{1}(u)=\left\{
\begin{array}
[c]{ll}%
1 & u\text{ \textrm{near} }b\\
0 & u\text{ \textrm{near} }a,
\end{array}
\right.
\]
a calculation shows that $\{(x_{1},0),(x_{2},0)\}$ is a GKN set for
$\widehat{T_{0}}$ in $H\oplus W$ where%
\[
x_{1}(u)=\left\{
\begin{array}
[c]{ll}%
u-b & u\text{ \textrm{near} }b\\
0 & u\text{ \textrm{near} }a
\end{array}
\right.  \text{ and }x_{2}(u)=\left\{
\begin{array}
[c]{ll}%
0 & u\text{ \textrm{near} }b\\
u-a & u\text{ \textrm{near} }a.
\end{array}
\right.
\]
We leave it to the reader to check that%
\begin{align*}
(\text{i})\text{ }[(x,a),(x_{1},0)]_{H\oplus W}  &  =x(b)-a\\
(\text{ii})\text{ }[(x,a),(x_{2},0)]_{H\oplus W}  &  =-x(a).
\end{align*}
From these equations, we see that the operator
\begin{align*}
\widehat{T}\mathcal{(}x,x(b))  &  =(-x^{\prime\prime},\alpha x(b)+x^{\prime
}(b))\\
\mathcal{D}(\widehat{T})  &  =\left\{  (x,x(b))\mid x\in\mathcal{D}%
(T_{1});x(a)=0\right\}
\end{align*}
is self-adjoint in $H\oplus W.\bigskip$

\noindent\textbf{Example 3.2 }By picking the partial GKN set $\{t_{1}\}$ for
$T_{0}$ in $H,$ where%
\[
t_{1}(u)=\left\{
\begin{array}
[c]{ll}%
0 & u\text{ \textrm{near} }b\\
u-a & u\text{ \textrm{near} }a,
\end{array}
\right.
\]
and the GKN set $\{(x_{1},0),(x_{2},0)\}$ for $\widehat{T_{0}}$ in $H\oplus
W,$ where%
\[
x_{1}(u)=\left\{
\begin{array}
[c]{ll}%
0 & u\text{ \textrm{near} }b\\
1 & u\text{ \textrm{near} }a
\end{array}
\right.  \text{ and }x_{2}(u)=\left\{
\begin{array}
[c]{ll}%
u-b & u\text{ \textrm{near} }b\\
0 & u\text{ \textrm{near} }a,
\end{array}
\right.
\]
the reader can check that the operator
\begin{align*}
\widehat{T}\mathcal{(}x,x^{\prime}(a))  &  =(-x^{\prime\prime},\alpha
x^{\prime}(a)+x(a))\\
\mathcal{D}(\widehat{T})  &  =\left\{  (x,x^{\prime}(b))\mid x\in
\mathcal{D}(T_{1});x(b)=0\right\}
\end{align*}
is self-adjoint in $H\oplus W.$

\subsubsection{Two Dimensional Extension Spaces}

For the last three examples, let $W=\mathbb{C}^{2}$ have the weighted inner
product
\[
\left\langle (z_{1},z_{2}),(z_{1}^{\prime},z_{2}^{\prime})\right\rangle
_{W}=\frac{z_{1}\overline{z_{1}^{\prime}}}{M}+\frac{z_{2}\overline
{z_{2}^{\prime}}}{N},
\]
where $M,N>0.$ Let $\{\xi_{1}=(\sqrt{M},0),\xi_{2}=(0,\sqrt{N})\}$ be a basis
for $W.$ The reader can check that the most general form of a self-adjoint
operator $B:W\rightarrow W$, using this inner product, has the matrix
representation
\begin{equation}
B=\left(
\begin{array}
[c]{cc}%
\alpha & \beta\\
\overline{\beta}N/M & \gamma
\end{array}
\right)  , \label{Example 3 - B}%
\end{equation}
where $\alpha,\beta\in\mathbb{R}$ and $\beta\in\mathbb{C}.\bigskip$

\noindent\textbf{Example 3.3 }\label{Example 3 - 1}Define $t_{i},x_{i}%
\in\mathcal{D}(T_{1})$ $(i=1,2)$ by%
\[
t_{1}(u)=\left\{
\begin{array}
[c]{ll}%
\sqrt{M} & u\text{ near }a\\
0 & u\text{ near }b
\end{array}
\right.  ,\qquad t_{2}(u)=\left\{
\begin{array}
[c]{ll}%
0 & u\text{ near }a\\
\sqrt{N} & u\text{ near }b
\end{array}
\right.  ,
\]%
\[
x_{1}(u)=\left\{
\begin{array}
[c]{ll}%
\sqrt{M}(u-a) & u\text{ near }a\\
0 & u\text{ near }b
\end{array}
\right.  ,\qquad x_{2}(u)=\left\{
\begin{array}
[c]{ll}%
0 & u\text{ near }a\\
\sqrt{N}(u-b) & u\text{ near }b.
\end{array}
\right.
\]
It is the case that $\{t_{1},t_{2}\}$ is a GKN set for $T_{0}$ in $H$ and
$\{(x_{1},0),(x_{2},0)\}$ is a GKN\ set for $\widehat{T}_{0}$ in $H\oplus W.$
Moreover, using $($\ref{Example 3 Symplectic Form H}$)$, we see that%
\[
\lbrack x,t_{1}]_{H}=x^{\prime}(a)\sqrt{M};\text{ }[x,t_{2}]_{H}=-x^{\prime
}(b)\sqrt{N};\text{ }[x,x_{1}]_{H}=-x(a)\sqrt{M};\text{ }[x,x_{2}%
]_{H}=x(b)\sqrt{N}%
\]
and%
\[
\Omega x=[x,t_{1}]\xi_{1}+[x,t_{2}]\xi_{2}=(x^{\prime}(a)M,-x^{\prime
}(b)N)\quad(x\in\mathcal{D}(T_{1})).
\]
In particular,
\[
\Omega x_{1}=(M^{3/2},0)\text{ and }\Omega x_{2}=(0,-N^{3/2}).
\]
With $z=(z_{1},z_{2})\in W,$ calculations show that%
\begin{equation}
0=[(x,z),(x_{1},0)]_{H\oplus W}=-x(a)\sqrt{M}+z_{1}\sqrt{M}
\label{Example 3 - BC1}%
\end{equation}
and%
\begin{equation}
0=[(x,z),(x_{2},0)]_{H\oplus W}=x(b)\sqrt{N}-z_{2}\sqrt{N}
\label{Example 3 - BC2}%
\end{equation}
yielding%
\[
z_{1}=x(a)\text{ and }z_{2}=x(b).
\]
That is, the boundary conditions expressed in $($\ref{Example 3 - BC1}$)$ and
$($\ref{Example 3 - BC2}$)$ yield continuity of functions in the domain of the
self-adjoint operator $\widehat{T}:\mathcal{D}(\widehat{T})\subset H\oplus
W\rightarrow H\oplus W$ defined by%
\[
\widehat{T}(x,(x(a),x(b)))=(-x^{\prime\prime},B(x(a),x(b))-(x^{\prime
}(a)M,-x^{\prime}(b)N))
\]%
\[
\mathcal{D}(\widehat{T})=\{(x,(x(a),x(b)))\mid x\in\mathcal{D}(T_{1})\},
\]
where $B$ is given in $($\ref{Example 3 - B}$)$ and%
\[
B(x(a),x(b)):=\left(
\begin{array}
[c]{cc}%
\alpha & \beta\\
\overline{\beta}N/M & \gamma
\end{array}
\right)  \left(
\begin{array}
[c]{c}%
x(a)\\
x(b)
\end{array}
\right)  .
\]
In this case, the setting $H\oplus W$ can be identified with the Hilbert
function space $L_{\sigma}^{2}[a,b],$ given by%
\[
L_{\sigma}^{2}[a,b]=\{f:[a,b]\rightarrow\mathbb{C}\mid f\text{ is Lebesgue
measurable on }[a,b]\text{ and }\left\Vert f\right\Vert _{H\oplus W}%
<\infty\},
\]
where $\left\Vert \cdot\right\Vert _{H\oplus W}$ is the norm generated by the
inner product
\[
\left\langle x,y\right\rangle _{H\oplus W}=\int_{a}^{b}x(u)\overline
{y}(u)du+\frac{x(a)\overline{y}(a)}{M}+\frac{x(b)\overline{y}(b)}{N}%
\]
and $d\sigma$ is the Lebesgue-Stieltjes measure generated by the distribution
function%
\[
\sigma(u)=\left\{
\begin{array}
[c]{ll}%
-\dfrac{1}{M}+a & u\leq a\\
u & a<u<b\\
\dfrac{1}{N}+b & u\geq b.
\end{array}
\right.
\]
\textbf{Example 3.4 }For this example, we switch the roles of $\{t_{1}%
,t_{2}\}$ and $\{x_{1},x_{2}\}$ which are defined in Example
\ref{Example 3 - 1}. Note that, in this case, $\{(t_{1},0),(t_{2},0)\}$ is a
GKN set for $\widehat{T}_{0}$ and $\{x_{1},x_{2}\}$ is a GKN set for $T_{0}.$
The calculations given in the previous example hold with the exception
\[
\Omega x=(-x(a)M,x(b)M)\quad(x\in\mathcal{D}(T_{1})).
\]
Again, with $z=(z_{1},z_{2}),$ we find that%
\[
0=[(x,z),(t_{1},0)]_{H\oplus W}=x^{\prime}(a)\sqrt{M}-z_{1}\sqrt{M}%
\]
and%
\[
0=[(x,z),(t_{2},0)]_{H\oplus W}=-x^{\prime}(b)\sqrt{N}+z_{2}\sqrt{N}%
\]
so that%
\[
z_{1}=x^{\prime}(a)\text{ and }z_{2}=x^{\prime}(b).
\]
In this case, the operator $\widehat{T}:\mathcal{D}(\widehat{T})\subset
H\oplus W\rightarrow H\oplus W,$ defined by
\[
\widehat{T}\mathcal{(}x,(x^{\prime}(a),x^{\prime}(b)))=(-x^{\prime\prime
},B(x^{\prime}(a),x^{\prime}(b))-(-x(a)M,x(b)M)),
\]
with domain
\[
\mathcal{D}(\widehat{T})=\left\{  (x,(x^{\prime}(a),x^{\prime}(b)))\mid
x\in\mathcal{D}(T_{1})\right\}
\]
is self-adjoint. In this case, for $x,y\in\mathcal{D}(\widehat{T}),$ the inner
product on $H\oplus W$ simplifies to the discrete Sobolev inner product
\begin{equation}
\left\langle x,y\right\rangle _{H\oplus W}=\int_{a}^{b}x(u)\overline
{y}(u)du+\frac{x^{\prime}(a)\overline{y}^{\prime}(a)}{M}+\frac{x^{\prime
}(b)\overline{y}^{\prime}(b)}{N}. \label{Example 3 - 2 IP}%
\end{equation}
Notice that, because of the derivatives in the discrete part of $($%
\ref{Example 3 - 2 IP}$),$ the closure of $\mathcal{D}(\widehat{T})$ is not a
function space and there is no positive Borel measure generating this inner
product.\bigskip

\noindent\textbf{Example 3.5 }For our last example, we consider a variation of
the last two examples. Indeed, define $t_{i},x_{i}\in\mathcal{D}(T_{1})$
$(i=1,2)$ by%
\[
t_{1}(u)=\left\{
\begin{array}
[c]{ll}%
\sqrt{M} & u\text{ near }a\\
0 & u\text{ near }b
\end{array}
\right.  ,\qquad t_{2}(u)=\left\{
\begin{array}
[c]{ll}%
0 & u\text{ near }a\\
\sqrt{N}(u-b) & u\text{ near }b
\end{array}
\right.  ,
\]%
\[
x_{1}(u)=\left\{
\begin{array}
[c]{ll}%
0 & u\text{ near }a\\
\sqrt{N} & u\text{ near }b
\end{array}
\right.  ,\qquad x_{2}(u)=\left\{
\begin{array}
[c]{ll}%
\sqrt{M}(u-a) & u\text{ near }a\\
0 & u\text{ near }b
\end{array}
\right.
\]
In this case $\{t_{1},t_{2}\}$ is a GKN set for $T_{1}$ and $\{(x_{1}%
,0),(x_{2},0)\}$ is a GKN set for $\widehat{T}_{0}.$ Moreover, a calculation
shows
\[
\Omega x=(x^{\prime}(a)M,x(b)N).
\]
With $z=(z_{1},z_{2}),$ the two boundary conditions%
\[
0=[(x,z),(x_{1},0)]_{H\oplus W}=-x^{\prime}(b)\sqrt{N}+z_{2}\sqrt{N}%
\]%
\[
0=[(x,z),(x_{2},0)]_{H\oplus W}=-x(a)\sqrt{M}+z_{1}\sqrt{M}%
\]
yield
\[
z_{1}=x(a)\text{ and }z_{2}=x^{\prime}(b).
\]
These calculation show that the operator $\widehat{T}$, given by
\[
\widehat{T}\mathcal{(}x,(x(a),x^{\prime}(b)))=(-x^{\prime\prime}%
,B(x(a),x^{\prime}(b))-(x^{\prime}(a)M,x(b)N))
\]
with domain
\[
\mathcal{D}(\widehat{T})=\left\{  (x,(x(a),x^{\prime}(b)))\mid x\in
\mathcal{D}(T_{1})\right\}  ,
\]
is self-adjoint in $H\oplus W.$ For each $x,y\in\mathcal{D}(\widehat{T}),$ the
`mixed' inner product in $H\oplus W$ reduces to%
\[
\left\langle x,y\right\rangle _{H\oplus W}=\int_{a}^{b}x(u)\overline
{y}(u)du+\frac{x(a)\overline{y}(a)}{M}+\frac{x^{\prime}(b)\overline{y}%
^{\prime}(b)}{N}.
\]
As in the last example, no positive Borel measure generates this inner product
and the closure of $\mathcal{D}(\widehat{T})$ in the topology from
$\left\langle \cdot,\cdot\right\rangle _{H\oplus W}$ is not a function space.

\end{document}